\documentclass [a4paper,10pt,oneside]{article}
\usepackage[italian, english]{babel}
\usepackage{amsmath}
\usepackage{amsfonts}
\usepackage{amssymb}
\usepackage{enumerate}							
\usepackage{setspace}
\usepackage{mathdots}
\usepackage{graphicx}
\usepackage{hyperref}
\usepackage{tikz}
	\usetikzlibrary{arrows}
\usepackage{wrapfig}

\usepackage{stmaryrd}
\DeclareSymbolFont{bbold}{U}{bbold}{m}{n}
\DeclareSymbolFontAlphabet{\mathbbm}{bbold}

\usepackage{framed}									
\usepackage{color}
\definecolor{lightgrey}{rgb}{1,0.95,0.95}
\definecolor{darkgrey}{rgb}{0.6,0.6,0.6}

\usepackage{xypic}

\newcommand{\BB}{\mathsf B}
\newcommand{\CC}{\mathsf C}
\newcommand{\DD}{\mathsf D}
\newcommand{\PP}{\mathsf P}
\newcommand{\QQ}{\mathsf Q}

\newcommand{\UU}{\mathsf U}

\newcommand{\0}{\mathbf 0}
\renewcommand{\1}{\mathbf 1}
\renewcommand{\2}{\mathbf 2}

\newcommand{\BBB}{\mathcal{B}}
\newcommand{\CCC}{\mathcal{C}}
\newcommand{\DDD}{\mathcal{D}}

\newcommand{\FFF}{\mathcal{F}}
\newcommand{\GGG}{\mathcal{G}}
\newcommand{\LLL}{\mathcal{L}}
\newcommand{\PPP}{\mathcal{P}}

\newcommand{\cc}{\mathfrak{c}}

\newcommand{\res}{\upharpoonright}					

\newcommand{\cp}[1]{\left( #1 \right)}

\newcommand{\Qp}[1]{\left\llbracket #1 \right\rrbracket}
\newcommand{\ap}[1]{\langle #1 \rangle}
\newcommand{\bp}[1]{\left\lbrace #1 \right\rbrace}
\newcommand{\vp}[1]{\left\lvert #1 \right\rvert}

\newcommand{\dirlim}{\varinjlim}
\newcommand{\invlim}{\varprojlim}
\newcommand{\rcslim}{\lim\limits_{\sf rcs}}

\DeclareMathOperator{\cof}{cof}

\DeclareMathOperator{\dom}{dom}

\DeclareMathOperator{\val}{val}

\DeclareMathOperator{\trcl}{trcl}
\DeclareMathOperator{\rank}{rank}
\DeclareMathOperator{\crit}{crit}
\DeclareMathOperator{\set}{set}
\DeclareMathOperator{\class}{class}
\DeclareMathOperator{\Coll}{Coll}
\DeclareMathOperator{\Add}{Add}

\DeclareMathOperator{\cd}{cd}
\DeclareMathOperator{\dd}{dd}
\DeclareMathOperator{\cpd}{cpd}

\newcommand{\RO}{\ensuremath{\text{{\sf RO}}}}
\newcommand{\ON}{\ensuremath{\text{{\sf ON}}}}
\newcommand{\SP}{\ensuremath{\text{{\sf SP}}}}
\newcommand{\SSP}{\ensuremath{\text{{\sf SSP}}}}
\newcommand{\ALL}{\ensuremath{\Omega}}

\DeclareMathOperator{\UP}{UP}

\DeclareMathOperator{\AD}{AD}
\newcommand{\ADF}{\AD(\Delta^0_1)}
\DeclareMathOperator{\ZF}{ZF}
\DeclareMathOperator{\ZFC}{ZFC}
\DeclareMathOperator{\NBG}{NBG}
\DeclareMathOperator{\MK}{MK}
\DeclareMathOperator{\CH}{CH}
\DeclareMathOperator{\GCH}{GCH}

\DeclareMathOperator{\RA}{RA}
\DeclareMathOperator{\wRA}{wRA}

\DeclareMathOperator{\FA}{FA}
\DeclareMathOperator{\BFA}{BFA}

\DeclareMathOperator{\MM}{MM}
\DeclareMathOperator{\BMM}{BMM}
\DeclareMathOperator{\AX}{AX}
\newcommand{\UnderTilde}[1]{{\setbox1=\hbox{$#1$}\baselineskip=0pt\vtop{\hbox{$#1$}\hbox to\wd1{\hfil$\sim$\hfil}}}{}}

\usepackage{amsthm}
\usepackage{fancyhdr}

\theoremstyle{plain}
	\newtheorem*{theorem*}{Theorem}
	\newtheorem{theorem}{Theorem}[section]
	\newtheorem{proposition}[theorem]{Proposition}
	\newtheorem{lemma}[theorem]{Lemma}
	\newtheorem{corollary}[theorem]{Corollary}

\theoremstyle{definition}
	\newtheorem{definition}[theorem]{Definition}

	\newtheorem{question}[theorem]{Question}
\theoremstyle{remark}
	\newtheorem{remark}[theorem]{Remark}

\begin{document}

	\title{Absoluteness via Resurrection}
	\author{Giorgio Audrito, Matteo Viale}
	\date{}
	\maketitle

	\begin{abstract}
		The resurrection axioms are forcing axioms introduced recently by Hamkins and Johnstone, 
		developing on ideas of Chalons and Veli\v ckovi\'c. We introduce a stronger form of resurrection 
		axioms (the \emph{iterated} resurrection axioms $\RA_\alpha(\Gamma)$ for a class of forcings 
		$\Gamma$ and a given ordinal $\alpha$), and show that $\RA_\omega(\Gamma)$ implies
		generic absoluteness for the first-order theory of $H_{\gamma^+}$ with respect to forcings in 
		$\Gamma$ preserving the axiom, where $\gamma=\gamma_\Gamma$ is a cardinal which 
		depends on $\Gamma$ ($\gamma_\Gamma=\omega_1$ if $\Gamma$ is any among the 
		classes of countably closed, proper, semiproper, stationary set preserving forcings).
		
		We also prove that the consistency strength of these axioms is below that of a Mahlo cardinal 
		for most forcing classes, and below that of a stationary limit of supercompact cardinals for 
		the class of stationary set preserving posets. Moreover we outline that simultaneous 
		generic absoluteness for $H_{\gamma_0^+}$ with respect to $\Gamma_0$ and for 
		$H_{\gamma_1^+}$ with respect to $\Gamma_1$ with 
		$\gamma_0=\gamma_{\Gamma_0}\neq\gamma_{\Gamma_1}=\gamma_1$ is in principle possible, 
		and we present several natural models of the Morse Kelley set theory where this phenomenon occurs 
		(even for all $H_\gamma$ simultaneously).
		Finally, we compare the iterated resurrection axioms (and the generic absoluteness results 
		we can draw from them) with a variety of other forcing axioms, and also 
		with the generic absoluteness results by Woodin and the second author. 
	\end{abstract}

	\section{Introduction} \label{sec:introduction}

	It is a matter of fact that forcing is one of the most powerful tools to produce consistency results in set theory: forcing axioms turn it into a powerful instrument to prove theorems. This is done by showing that a statement $\phi$ follows from an extension $T$ of $\ZFC$ if and only if $T$ proves that $\phi$ is consistent by means of a certain type of forcing.
	These types of results are known in the literature as generic absoluteness results and have the general form of a completeness theorem for some $T\supseteq\ZFC$ with respect to the semantics given by Boolean-valued models and first-order calculus. More precisely, generic absoluteness theorems fit within the following general framework:
	
		\begin{quote}
			Assume $T$ is an extension of $\ZFC$, $\Theta$ is a family of first-order formulae in the language of set theory and $\Gamma$ is a certain  definable class of forcing notions.
			Then the following are equivalent for any $\phi \in \Theta$ and $S\supseteq T$:
			\begin{enumerate}
				\item $S$ proves $\phi$.
				\item $S$ proves that there exists a forcing $\BB \in\Gamma$ such that $\BB$ forces $\phi$ and $T$ jointly.
				\item $S$ proves that $\BB$ forces $\phi$ for all forcings $\BB\in\Gamma$ such that $\BB$ forces $T$.
			\end{enumerate}
		\end{quote}
		
	We say that a definable structure $M$ is generically invariant with respect to forcings in 
	$\Gamma$ and parameters in $X\subset M$, when the above situation occurs with $\Theta$ 
	being the first-order theory of $M$ with parameters in $X$.
	A brief overview of the main known generic absoluteness results is the following:
	
	\begin{itemize}
			\item
			Shoenfield's absoluteness theorem is a generic absoluteness result for $\Theta$ the family of $\Sigma^1_2$-properties with real parameters, $\Gamma$ the class of all forcings, $T=\ZFC$.
	
			\item
			The pioneering ``modern'' generic absoluteness results are Woodin's proofs of the invariance under set forcings of the first-order theory of $L(\ON^\omega)$ with real parameters in $\ZFC+$\emph{class-many Woodin cardinals which are limit of Woodin cardinals} \cite[Thm. 3.1.2]{larson:stationary_tower} and of the invariance under set forcings of the family of $\Sigma^2_1$-properties with real parameters in the theory $\ZFC+\CH+$\emph{class-many measurable Woodin cardinals} \cite[Thm. 3.2.1]{larson:stationary_tower}.
	
			Further results pin down the exact large cardinal strength of the assertion that $L(\mathbb{R})$ is generically invariant with respect to certain classes of forcings (among others see~\cite{bosch:forcing_solovay, bosch:solovay_forcing, bosch:projective_absoluteness, prisco:partition_reals, neeman:absoluteness, schindler:proper_remarkable1, schindler:proper_remarkable2}).
	
			\item
			The bounded forcing axioms $\BFA_{\omega_1}(\Gamma)$ for $\Gamma$ among the classes 
			of proper, semiproper, stationary set preserving forcings are equivalent to the statement that 
			generic absoluteness holds for $T=\ZFC$ and $\Theta$ 
			the class of $\Sigma_1$-formulae with parameters in $\PPP(\omega_1)$, 
			as shown in \cite{bagaria:bfa_absoluteness}.
	
			\item
			Recently, Hamkins and Johnstone \cite{hamkins:resurrection_uplifting} introduced the 
			resurrection axioms $\RA(\Gamma)$ and Viale \cite{viale:mm_revisited} showed that 
			these axioms produce generic absoluteness for  $\Theta$ the $\Sigma_2$-theory with 
			parameters of $H_{\cc}$, $T=\ZFC+\RA(\Gamma)$, $\Gamma$ any of the standard classes of 
			forcings closed under two step iterations.
	
			\item
			Viale introduced the forcing axiom $\MM^{+++}$ (a natural strenghtening of Martin's maximum 
			$\MM$) and proved that $L(\ON^{\omega_1})$ with parameters in $\PPP(\omega_1)$
			is generically invariant with respect to stationary set preserving ($\SSP$) forcings for
			\[
			T=\ZFC+\MM^{+++}+\text{there are class-many superhuge cardinals}.
			\]
	\end{itemize}
	
	Motivated by the latter results as well as by the work of Tsaprounis \cite{tsaprounis:on_resurrection}, 
	we introduce over the theory $\MK$ (i.e., the Morse-Kelley set theory with sets and classes) 
	a new natural class of forcing axioms: the iterated resurrection axioms $\RA_{\alpha}(\Gamma)$ 
	of increasing strength as $\alpha$ runs through the ordinals, 
	with $\Gamma$ a definable class of forcing notions. We also remark that for most classes $\Gamma$, 
	$\RA_1(\Gamma)$ is a slight strengthening of the resurrection axiom $\RA(\Gamma)$ 
	recently introduced by Hamkins and Johnstone in \cite{hamkins:resurrection_uplifting}.
	
	We are able to prove over $\MK$ the consistency relative to large cardinal axioms of the axioms 
	$\RA_{\alpha}(\Gamma)$ for (essentially) any definable class of forcing notions $\Gamma$ which 
	is \emph{weakly iterable} (i.e., 
	closed under two-step iterations and having an iteration strategy allowing limits of arbitrary 
	length to be in $\Gamma$, 
	see  Def. \ref{def:weak_iterable}).
	The latter is a property of classes of forcings which holds for most standard classes such as  
	the class $\Omega$ of all forcing notions, the classes $\Gamma_\kappa$ of ${<}\kappa$-closed forcings for any 
	regular cardinal $\kappa$, the classes of 
	 axiom-$A$, proper, semiproper ($\SP$), stationary set preserving ($\SSP$) forcings
	 (the latter class being weakly iterable only in 
	 the presence of sufficiently strong large cardinal axioms).
	
	The main motivation leading to the axioms $\RA_{\alpha}(\Gamma)$ with $\alpha\geq\omega$ is 
	the following generic absoluteness result over the theory $\MK$.
	
	\begin{theorem*}
		Let $V$ be a model of $\MK$, $\Gamma$ be a definable class of forcing notions 
		closed under two step iterations.
		Assume that there exists a largest cardinal $\gamma=\gamma_\Gamma$ 
		which is preserved by any forcing in $\Gamma$, and that all ${<}\gamma$-closed forcings belong to 
		$\Gamma$.
		Then 
		\[
		H_{\gamma^+}^V \prec H_{\gamma^+}^{V^\BB}
		\]
		whenever $\RA_\omega(\Gamma)$ holds in $V$ and $\BB \in \Gamma$ forces 
		$\RA_\omega(\Gamma)$. 
	\end{theorem*}
	
	This is a generic absoluteness result for $T=\MK+\RA_\omega(\Gamma)$ and $\Theta$ the first-order theory of $H_{\gamma_\Gamma^+}$ with parameters. We also prove that the consistency strength of the axioms $\RA_\alpha(\Gamma)$ is below that of a Mahlo cardinal for all relevant $\Gamma$ and for all $\alpha$ except for $\Gamma=\SSP$, in which case our upper bound is below a Mahlo cardinal which is a limit of supercompact cardinals.
		
	Furthermore, the axioms $\RA_\alpha(\Gamma)$ fit naturally within the hierarchy given by a variety of other
	forcing axioms, for example: $\MM^{+++} \Rightarrow \RA_\alpha(\SSP)$ and 
	$\RA_\alpha(\Gamma) \Rightarrow \RA(\Gamma)+\BFA(\Gamma)$ for most $\Gamma$ and for all $\alpha>0$.
	We also outline that $\RA_\omega(\ALL)$ (where $\ALL$ denotes the class of all forcings) is a natural weakening of Woodin's generic absoluteness results for $L(\mathbb{R})$, and $\RA_\alpha(\SSP)$ a natural strenghtening of $\BMM$ for all $\alpha>0$.
	
	These results cannot be formulated in $\ZFC$ alone since the iterated resurrection axioms $\RA_{\alpha}(\Gamma)$ are in essence second-order statements. However, it could be possible that some natural theory strictly weaker than $\MK$ (e.g., $\NBG$ together with a truth predicate) would suffice to carry out all the arguments of the present paper.
	On the other hand the axioms $\RA_n(\Gamma)$ for $n < \omega$ can also be formulated as a $\ZFC$ first-order sentence equivalent to their second-order definition in $\MK$, and the axiom $\RA_\omega(\Gamma)$ can be formulated as a first-order axiom schema $\bp{\RA_n(\Gamma): ~ n < \omega}$ in the $\ZFC$-setting (assuming that $\Gamma$ is definable in $\ZFC$, which is true for all the relevant cases).
	
	Altogether these results show the effectiveness of the axioms $\RA_\omega(\Gamma)$ both on the \emph{premises} side (low consistency strength, natural generalization of well-known axioms) and on the \emph{consequences} side (generic invariance of $H_{\gamma_\Gamma^+}$).
	However, the axioms $\RA_\omega(\Gamma)$ are pairwise incompatible for most different choices of $\Gamma$ with the same $\gamma_\Gamma$: for example, $\RA_\omega(\SSP)+\textit{there are class-many supercompact cardinals}$ implies that canonical functions are a dominating family, 
	whereas $\RA_\omega(\text{proper})$ implies that they are not dominating \cite[Fact 5.1]{viale:mm_revisited}. 
	These two latter assertions are expressible as $\Pi_2$ (or $\Sigma_2$) 
	properties on $H_{\omega_2}$ with parameter $\omega_1$.
	
	We are eager to accept the position that new ``natural'' axioms should be added to set theory in order to settle the undecidability phenomenon, and we consider the iterated resurrection axioms among the candidates for these new ``natural'' axioms (we will expand on this bold assertion in the concluding part of this paper). 
	Even assuming this philosophical position regarding mathematical truth,
	the above results outline that  we need some other philosophical criterion to select which among the various consistent axioms $\RA_\alpha(\Gamma)$ with a fixed $\gamma_\Gamma$ is the most reasonable candidate to supply a ``new natural axiom for set theory''. 
	Likewise we also need some other guidance to select for each cardinal $\kappa$ which class 
	$\Gamma$ with $\gamma_\Gamma=\kappa$ could be the most natural largest class for which 
	$\RA_\alpha(\Gamma)$ can be predicated (if such a class exists at all). 
	Towards this aim, we remark the following three special cases:
	
	\begin{itemize}
		\item The axiom $\RA_\omega(\ALL)$ is consistent and provides generic invariance for the theory of
		$H_{\omega_1}$ with respect to any forcing preserving $\RA_\omega(\ALL)$ 
		(recall that $\ALL$ for us denotes the class of all set-sized forcings).
		
		\item If we focus on forcing classes $\Gamma$ whose corresponding $\gamma_\Gamma$ is 
		$\omega_1$, there is a unique largest class 
		(the class of stationary set preserving posets $\SSP$) which contains all the possible classes 
		$\Gamma$ for which the axiom $\RA_\omega(\Gamma)$ is consistent. 
		Thus $\RA_\omega(\SSP)$ gives the strongest form of generic absoluteness for
		$H_{\omega_2}$ 
		which can be instantiated by means of the iterated resurrection axioms for these type of 
		forcing classes $\Gamma$. Moreover we will show that 
		$\RA_\omega(\SSP)$ and $\RA_\omega(\ALL)$ are jointly
		consistent.
		
		\item On the other hand we are still in the dark regarding which could be (or even if it can exists) 
		the most natural largest class 
		$\Gamma$ having $\gamma_\Gamma>\omega_1$ for which $\RA_\omega(\Gamma)$ is consistent.
		Nonetheless if we consider the forcing classes of ${<}\omega_\alpha$-closed forcings for $\alpha \in \ON$, 
		the corresponding resurrection axioms are all pairwise compatible. For example, from a Mahlo cardinal 
		(existing in a model of $\RA_\omega(\ALL)+\RA_\omega(\SSP)$) 
		it is possible to obtain the consistency of the theory:
		
		\emph{
		\begin{quote}
		$\MK + \GCH^{>\omega_1} + \RA_\omega(\ALL)+\RA_\omega(\SSP)+$\\
		$\RA_\alpha({<}\kappa\text{-closed})$ for 
		all cardinals $\kappa>\omega_1$ and ordinals $\alpha$. 
		\end{quote}
		}
		
		\smallskip
		
		This gives a global and uniform generic absoluteness result, i.e., in this model, given any forcing 
		$\BB$, we have that $H_{\kappa^+} \prec H_{\kappa^+}^\BB$ 
		for any $\kappa$ such that $\BB$ is ${<}\kappa$-closed.
	\end{itemize}
	
	Even though the class of ${<}\omega_\alpha$-closed forcing is narrow and (some notion of) 
	$\omega_\alpha$-(semi)proper forcing would be preferable, our results highlights a strong connection of 
	the theory of iterations with generic absoluteness:
	a main outcome of the present work gives that any reasonable preservation theorem for iterations along
	a forcing class $\Gamma$ (i.e., a theorem asserting that suitable iterations of forcings 
	in $\Gamma$ produce limits which are as well in $\Gamma$)
	entails the consistency of the axiom $\RA_\omega(\Gamma)$ 
	and yields a generic absoluteness result for the same class $\Gamma$.
	
	Compared to the generic absoluteness result obtained in \cite{viale:category_forcing}, 
	the present results for $\Gamma = \SSP$ are weaker, since they regard the structure 
	$H_{\omega_2}$ instead of the larger structure $L(\ON^{\omega_1})$. 
	On the other hand, the consistency of $\RA_\alpha(\Gamma)$ for $\Gamma\neq\SSP$ is obtained from 
	(in most cases much) weaker large cardinal hypothesis and the results are more general 
	since they also apply to interesting natural choices of $\Gamma \neq \SP, \SSP$ such as 
	$\ALL$, axiom-$A$ and proper.
	Moreover the arguments we employ to prove the consistency of $\RA_\alpha(\Gamma)$ 
	are considerably simpler than the arguments developed in~\cite{viale:category_forcing}.

	Finally, a by-product of our results is that the theory $T = \MK+\RA_{\omega}(\ALL)$ 
	is consistent relative to the existence of a Mahlo cardinal and makes the theory of projective sets of 
	reals generically invariant with respect to any forcing which preserves $T$. 
	Notice that we can force $T$ to hold in a generic extension
	of $L_{\delta+1}$ with $\delta$ a Mahlo cardinal in $L$, and 
	projective determinacy cannot hold in this generic extension, since $0^\sharp$ does not exists in this model.	
	Thus $\MK+\RA_{\omega}(\ALL)$ is consistent with the failure of projective determinacy. 
	This shows that the request of generic absoluteness for projective sets of reals 
	(with respect to forcings preserving $\MK+\RA_{\omega}(\ALL)$) 
	is a natural weakening of Woodin's generic absoluteness 
	for projective sets of reals with respect to $\ZFC+$\emph{large cardinals} and \emph{all} forcing notions.	

	The remaining part of Section~\ref{sec:introduction} recalls some standard terminology. Section \ref{sec:backgrounds} presents determinacy axioms for class games and the concept of weak iterability, which will be necessary in later sections.
	Section~\ref{sec:absoluteness} introduces the definition of the iterated resurrection axioms together with their basic properties, and proves Theorem \ref{thm:absoluteness} stating that the axioms
	$\RA_\alpha(\Gamma)$ for $\alpha$ infinite produce a generic absoluteness result for $H_{\gamma^+}$ (with 
	$\gamma=\gamma_\Gamma$ depending on $\Gamma$). 
	Section~\ref{sec:uplifting} develops the necessary technical devices for the consistency proofs of the axioms $\RA_\alpha(\Gamma)$. These proofs are carried out in Section~\ref{sec:consistency} adapting the consistency proofs for the resurrection axioms introduced in~\cite{hamkins:resurrection_uplifting} to our new setting. Section \ref{sec:resurrection_open} outlines the main possible directions of further research on the topic.

\subsection{Notation} \label{ssec:notation}

	Following standard set-theoretic terminology,
	 $\trcl(x)$, $\rank(x)$ denote respectively the transitive closure and the rank of a given set $x$. 
	 $V_\alpha$ is the set of $x$ such that $\rank(x) < \alpha$ and $H_\kappa$ is the set of $x$ such that 
	 $\vp{\trcl(x)} < \kappa$. We use $\PPP(x)$, $[x]^\kappa$, $[x]^{{<}\kappa}$ to denote the powerset of $x$, 
	 the set of subsets of $x$ size $\kappa$, and the set of subsets of $x$ of size less than $\kappa$. 
	 Given a set $M$, $\pi_M:M\to Z_M$ denotes its Mostowski collapse map onto a transitive set $Z_M$.
	 The notation $f : A \to B$ is improperly used to denote partial functions in $A \times B$, ${}^AB$ to 
	 denote the collection of all such (partial) functions, and $f[A]$ to denote the pointwise image of $A$ through $f$. 
	 We use $s^\smallfrown t$ for sequence concatenation and $s^\smallfrown x$ where 
	 $x$ is not a sequence as a shorthand for $s^\smallfrown \ap{x}$. 
	 We use $t \vartriangleleft s$ to denote that $s = t \res (\vp{t}-1)$. 
	 $\CH$ denotes the continuum hypothesis and $\cc$ the cardinality of the continuum itself. 
	 We prefer the notation $\omega_\alpha$ instead of $\aleph_\alpha$ for cardinals.
	
	Let $\LLL^2$ be the language of set theory with two sorts of variables, one for sets and one for classes.
	We work with theories in the language $\LLL^2$ extending the G\"odel-Bernays system $\NBG$ and in most cases also the
	Morse-Kelley system $\MK$ (see \cite[Sec. II.10]{kunen:set_theory} for the axioms of $\NBG$ and \cite{antos:class_forcing} for $\MK$),
	and use lower-case letters for set variables and parameters and upper-case letters for class variables and parameters. We say that a formula of $\mathcal{L}^2$ is first-order if its quantifiers range only over sets.
	We remark that $\MK$ has a reasonable consistency strength (below that of an inaccessible cardinal) and 
	is a natural strengthening of $\ZFC$ (or of its equivalent $\mathcal{L}^2$-formulation $\NBG$), 
	since it asserts second-order properties that are true for natural models of $\ZFC$ 
	(models of the kind $V_\delta$ with $\delta$ inaccessible).
	
	In general we 
	identify an $\LLL^2$-model $N=\ap{\set(N),\class(N)}$ of $\NBG$ 
	with its underlying collection of classes $\class(N)$: 
	e.g., if $\kappa$ is inaccessible and $N=\ap{V_\kappa,V_{\kappa+1}}$, we just denote $N$ by $V_{\kappa+1}$. 
	We recall that from $\class(N)$ we can reconstruct whether $x \in \set(N)$ 
	via the formula $\exists y \in \class(N) ~ x \in y$.

	We use $M \prec_n N$ to denote that $(M,\in)$ is a $\Sigma_n$-elementary substructure of $(N,\in)$. 	
	Given an elementary embedding $j: V \to M$ with $M$ transitive, $\crit(j)$ denotes the critical point of $j$. 
	We say that $j$ is $\lambda$-supercompact iff ${}^\lambda M \subseteq M$, and that a cardinal 
	$\kappa$ is \emph{supercompact} iff for every $\lambda$ there exists a $\lambda$-supercompact 
	elementary embedding with critical point $\kappa$.
	Our reference text for large cardinals is \cite{kanamori:higher_infinite}, while for forcing 
	axioms is~\cite[Ch. 3]{bekkali:set_theory}.
	
	We follow Jech's approach \cite{jech:set_theory} to forcing via Boolean-valued models. 
	The letters $\BB, \CC, \DD,\ldots$ are used for set-sized complete Boolean algebras, $\0, \1$ 
	denote their minimal and maximal element, and $\2$ denotes the $2$-element Boolean algebra. We denote the Boolean-valued model obtained from $V$ and $\BB$ as $V^\BB$. If $\ap{V,\CCC}$ is a model of $\MK$, we let:
	\[
	V^{\BB}=\bp{f:V^\BB\to\BB: f\in V}, \quad \CCC^{\BB}=\bp{f:V^\BB\to\BB: f\in \CCC}.
	\]
	$\dot{x}$ (resp. $\dot{X}$) denotes an element ($\BB$-name) of $V^\BB$ (resp. $\CCC^\BB$).
	$\check{x}$ (resp. $\check{X}$)  denotes the canonical $\BB$-name for a set $x \in V$ (resp. class $X \in \CCC$) in the Boolean-valued model $V^\BB$ (resp. $\CCC^\BB$).
	$\Qp{\phi}_\BB$ is the truth-value of the formula $\phi$. We sometimes confuse $\BB$-names with their defining properties: 
	for example, given a collection of $\BB$-names $\bp{\dot{x}_\alpha : ~ \alpha < \gamma}$, 
	we confuse $\bp{\dot{x}_\alpha : ~ \alpha < \gamma}$ with a $\BB$-name $\dot{x}$ such that for all 
	$\dot{y}\in V^{\BB}$, $\Qp{\dot{y}\in\dot{x}}_\BB = \Qp{\exists \alpha < \check{\gamma} ~~ 
	\dot{y}=\dot{x}_\alpha}_\BB$. 
	When we believe this convention may confuse the reader we shall be explicitly more careful.
	
	When convenient we also use the generic filters approach to forcing. The letters $G$, $H$ will be 
	used for generic filters over $V$, $\dot{G}_\BB$ denotes the canonical name for a generic filter for 
	$\BB$, $\val_G(\dot{x})$ the valuation map on names by the generic filter $G$, $V[G]$ the generic extension of $V$ by $G$. 
	Given $G$ $V$-generic for $\BB\in V$:
	\[
	V[G]=\bp{\val_G(\dot{x}) : ~ \dot{x}\in V^\BB}, \quad \CCC[G]=\bp{\val_G(\dot{X}) : ~ \dot{X}\in \CCC^\BB}.
	\]
	A key result we freely use throughout the paper is that $\MK$ (respectively $\NBG$) 
	is preserved by any set-sized forcing over an $\LLL^2$-model of $\MK$ (respectively $\NBG$), 
	\cite[Thm. 23]{antos:class_forcing}:
	if $\ap{V,\CCC}\models\MK$ (resp. $\ap{V,\CCC}\models\NBG$), then
	$\ap{V[G],\CCC[G]}\models\MK$ (resp. $\ap{V[G],\CCC[G]}\models\NBG$) for all set-sized forcings $\BB\in V$
	and all $V$-generic filters $G$ for $\BB$. 
	Let $\phi$ be a formula in $\LLL^2$. We write $V^\BB \models \phi$ to denote that $\phi$ holds in all generic extensions 
	$\ap{V[G],\CCC[G]}$ with $G$ $V$-generic for $\BB$.
		
	We use $\Coll(\kappa,{<}\lambda)$ for the L\'evy collapse that generically adds a surjective function from $\kappa$ to any $\gamma < \lambda$, $\Add(\kappa,\lambda)$ for the ${<}\kappa$-closed poset that generically adds $\lambda$ many subsets to $\kappa$. We prefer the notation ``$X$ has the ${<}\kappa$-\emph{property}'' for all properties that are defined in terms of $\forall \gamma < \kappa ~ \phi(\gamma, X)$ for some formula $\phi$. In all such cases we explicitly avoid the notation ``$\kappa$-\emph{property}'' and use ${<}\kappa^+$-\emph{property} instead. In general we feel free to confuse a partial order $P$ with its Boolean completion $\RO(P)$ and a Boolean algebra $\BB$ with the partial order $\BB^+$ given by its positive elements. Once again, when we believe that this convention may generate misunderstandings we shall be explicitly more careful.
	
	Given a cardinal $\kappa$ definable by some formula $\phi(x)$ in one free variable, we let $\dot{\kappa}$ be a $\BB$-name such that $\Qp{\phi(\dot{\kappa})}_{\BB}=\1$ and improperly write $H_{\kappa}^\BB$ for a $\BB$-name whose interpretation in $V[G]$ for any $V$-generic filter $G$ is the structure $H_{\val_G(\dot{\kappa})}^{V[G]}$ (e.g., if $\kappa$ is $\omega_1$ and $\BB$ is the Boolean completion of $\Coll(\omega,\omega_1)$, we write $H_{\omega_1}^\BB$ for $H_{\check{\omega}_2}^\BB$).
	We use $H_{\kappa}^\BB \prec H_{\kappa}^\CC$ (respectively $H_{\kappa} \prec H_{\kappa}^\BB$) to denote that $\BB$ is a complete subforcing of $\CC$ and for all $G$ $V$-generic for $\BB$ and $H$ $V[G]$-generic for $\CC/_G$, $H_{\kappa}^{V[G]} \prec H_{\kappa}^{V[G][H]}$.

	\section{Backgrounds} \label{sec:backgrounds}


\subsection{Second-order elementarity and class games} \label{sec:class_games}

In order to be able to properly formulate the iterated resurrection axioms and prove their consistency, we need to 
carefully examine the syntactic complexity of several basic notions regarding proper classes. Most of our results
will be formulated as second order statements on the universe of sets. Hence we need to be able to 
handle different logically equivalent formulations of a variety of second order properties; properties which we also 
need to show to be of low logical complexity with respect to second order logic. 

	Let $T$ be a theory in the language $\LLL^2$ extending $\NBG$ that provably holds in $V_{\delta+1}$ with $\delta$ inaccessible. Let $\Delta^1_1(T)$ denote the formulae in $\LLL^2$ with set parameters that are provably equivalent modulo $T$ both to a $\Sigma^1_1$ formula (i. e. a formula with one existential class quantifier and set parameters) and to a $\Pi^1_1$ formula (with one universal class quantifier and set parameters).
	
	Let $N$, $M$ be models of $T$. We say that $N \prec_{\Delta^1_1(T)} M$ iff $\set(N) \subseteq \set(M)$ and all 
	$\Delta^1_1(T)$-formulae with set parameters in $N$ hold in $N$ if and only if they also hold in $M$. 
	Similarly, we say that $N \equiv_{\Delta^1_1(T)} M$ iff $N \prec_{\Delta^1_1(T)} M \prec_{\Delta^1_1(T)} N$. 
	Note that $N \prec_{\Delta^1_1(T)} M$ does not require $\class(N) \subseteq \class(M)$: we also use this 
	notion when $N, M \prec V_{\kappa+1}$ for some inaccessible cardinal $\kappa$ and 
	$N \cap V_\kappa = V_\delta = M \cap V_\kappa$ for some inaccessible cardinal $\delta<\kappa$ (e.g., Lemma~\ref{lem:statdelta21}). 
	In this situation $N, M \not\subseteq V_{\delta+1}$, but we are still able to say that $N, M$ are 
	$\Delta^1_1(T)$-elementary in $V_{\delta+1}$ or that $N \equiv_{\Delta^1_1(T)} M$ 
	(since only set parameters are considered).

	The $\Delta^1_1(T)$-formulae are interesting for their absolute behavior with respect to models of $T$ with the same sets.

	\begin{proposition} \label{prop:delta21abs}
	Let $T$ be a theory extending $\NBG$, $N, M$ be models of $T$ with $\set(N) = \set(M)$ and $N \subseteq M$. Then $N \equiv_{\Delta^1_1(T)} M$.
	\end{proposition}
	\begin{proof}
		Let $X = \set(N) = \set(M)$, and $\phi \equiv \exists C ~ \psi_1(C, \vec{p}) \equiv \forall C ~ \psi_2(C, \vec{p})$ be a $\Delta^1_1(T)$ formula, with $\vec{p} \in X$ and $\psi_i$ first-order for $i=1,2$. Then,
		\[
		\begin{array}{rl}
		N \models \phi &\Rightarrow \exists C \in N ~ \ap{X, C} \models \psi_1(C, \vec{p}) \\
		&\Rightarrow \exists C \in M ~ \ap{X, C} \models \psi_1(C, \vec{p}) \Rightarrow \\
		M \models \phi &\Rightarrow \forall C \in M ~ \ap{X, C} \models \psi_2(C, \vec{p}) \\
		&\Rightarrow \forall C \in N ~ \ap{X, C} \models \psi_2(C, \vec{p}) \Rightarrow N \models \phi 	\qedhere
		\end{array}
		\]
	\end{proof}

	\begin{corollary} \label{cor:delta21abs}
		Let $T$ be a theory extending $\NBG$ that holds in $V_{\delta+1}$ with $\delta$ inaccessible. Let $N, M$ be models of $T$ with $\set{(N)} = \set{(M)} = V_\delta$. Then $N \equiv_{\Delta^1_1(T)} M$.
	\end{corollary}
	\begin{proof}
		$V_{\delta+1}$ is a model of $T$ containing both of the transitive collapses $Z_N$, $Z_M$ of $M$ and $N$,
		 hence by Lemma \ref{prop:delta21abs}, 
		 $N \cong Z_N \equiv_{\Delta^1_1(T)} V_{\delta+1} \equiv_{\Delta^1_1(T)} Z_M \cong M$.
	\end{proof}

	Proposition \ref{prop:delta21abs} and its corollary tells us that for any inaccessible $\delta$ the truth-value of a $\Delta^1_1(T)$-formula does not depend on the choice of the particular $T$-model whose family of sets is $V_\delta$. Thus we mostly focus on $\NBG$-models of the kind $V_{\delta+1}$.
	
	We are also interested in a more restrictive class of formulae, which we call \emph{canonical $\Delta^1_1(T)$-formulae}.
	
	\begin{definition} \label{def:canonicalformula}
		Let $\phi(x_1,\ldots,x_n)$ be an $\LLL^2$-formula with set variables $x_1,\ldots,x_n$. 
		$\phi$ is a \emph{canonical $\Delta^1_1(T)$-formula} iff there are first-order formulae $\psi_1(x_1,\ldots,x_n,Y)$, $\psi_2(x_1,\ldots,x_n,Y)$ with a class parameter $Y$ such that:
		\[
		\begin{split}
			T &\models \forall x_1,\dots,x_n ~ \exists! Y ~ \psi_2(x_1,\ldots,x_n,Y) \\
			\phi(x_1,\ldots,x_n)
			&\equiv_T \exists Y ~ \cp{\psi_2(x_1,\dots,x_n,Y) \wedge \psi_1(x_1,\dots,x_n,Y)}  \\
			&\equiv_T \forall Y ~ \cp{\psi_2(x_1,\dots,x_n,Y) \rightarrow \psi_1(x_1,\dots,x_n,Y)}.
		\end{split}
		\]
	\end{definition}
	
	In other words, an $\LLL^2$-formula is a canonical $\Delta^1_1(T)$-formula if it can be expressed as a first-order property $\psi_1(x_1,\dots,x_n,C)$ of a class parameter $C$, which is uniquely defined by a first-order formula $\psi_2(x_1,\dots,x_n,Y)$ with a class variable $Y$ (in theory $T$).
	By its very definition a canonical $\Delta^1_1(T)$-property $\phi$ is indeed a $\Delta^1_1(T)$-property. On the other hand canonical $\Delta^1_1(T)$-formulae are closed under connectives and set-quantifiers, hence they are best suited to handle syntactic manipulations of formulae.
	
	\begin{proposition} \label{prop:canonicalformula}
		Let $\phi^i(x_0,\ldots,x_n)$ be canonical $\Delta^1_1(T)$-formulae for $i=1,2$. 
		Then also $\neg \phi^i$, $\phi^1 \wedge \phi^2$ and $\exists x_0 ~ \phi^i$ are 
		canonical $\Delta^1_1(T)$-formulae.
	\end{proposition}
	\begin{proof}
		Assume that $\phi^i$ is defined through $\psi^i_1(x_0,\dots,x_n,Y)$ and 
		$\psi^i_2(x_0,\dots,x_n,Y)$ as in the above definition. 
		Given a class $Y$, let $Y[x]=\bp{y: \ap{ x,y }\in Y}$.\footnote{Here to simplify matters we abuse of the reader mixing syntactic notions with
				their model theoretic interpretations. The key fact is that the map $\ap{x,Y}\mapsto Y[x]$ is 
				$\Sigma_0$-definable in $\NBG$.}
		The following holds:
		\begin{itemize}
			\item $\neg \phi^i(x_0,\dots,x_n)$ is defined through the same ``unique class'' $\psi'_2 \equiv \psi^i_2$ and opposite property $\psi'_1 (x_0,\dots,x_n,Y)
			\equiv \neg \psi^i_1(x_0,\dots,x_n,Y)$;
			\item $\phi^1(x_0,\dots,x_n) \wedge \phi^2(x_0,\dots,x_n)$ is defined through:
			\[
			\begin{split}
				\qquad \psi'_2(x_0,\dots,x_n,Y) &\equiv \psi^1_2(x_0,\dots,x_n,Y[0]) \wedge \psi^2_2(x_0,\dots,x_n,Y[1]) \wedge \dom(Y) = 2 \\
				\qquad \psi'_1(x_0,\dots,x_n,Y) &\equiv \psi^1_1(x_0,\dots,x_n,Y[0]) \wedge \psi^2_1(x_0,\dots,x_n,Y[1])
			\end{split}
			\] 
			so that the required unique class $Y=\bp{\ap{0,x}:x\in Y_0}\cup\bp{\ap{1,y}:y\in Y_1}$ 
			is a ``gluing'' of the two classes $Y_0,Y_1$ which work for $\phi^1$, $\phi^2$;
			\item $\exists x ~ \phi^i(x,x_1,\dots,x_n)$ is defined through:
			\[
			\begin{split}
				\psi'_2(x_0,\dots,x_n,Y) &\equiv \forall x ~ \psi^i_2(x,x_1,\ldots,x_n,Y[x]) \\
				\psi'_1(x_0,\dots,x_n,Y) &\equiv \forall x ~ \psi^i_1(x,x_1,\ldots,x_n,Y[x])
			\end{split}
			\] 
			so that the required unique class $Y=\bp{\ap{x,y}:y\in Y_x}$ 
			glues together all the classes $Y_x$ satisfying 
			$\psi^i_2(x,x_1,\dots,x_n,Y_x)$ for every possible value of $x$. \qedhere
		\end{itemize}
	\end{proof}

	In the remainder of this paper, we will need to prove that certain statements about class games are $\Delta^1_1(T)$ for a suitable theory $T$: what we will actually prove is that they are \emph{canonically} $\Delta^1_1(T)$. 
	In order to justify the choice of a theory $T$ we need to introduce clopen games on class trees. 
	Our reference text for the basic notions and properties of games is \cite[Sec. 20.A]{kechris:descriptive}.

	We consider well-founded trees $U$ as collections of finite sequences ordered by inclusion and 
	closed under initial segments, such that there exists no infinite chain (totally ordered subset) in $U$. 
	Recall that $t \vartriangleleft s$ for $s, t \in U$ denotes that 
	$s = t \res (\vp{t}-1)$ ($t$ is obtained extending $s$ with one more element).

	The clopen game on the well-founded tree $U$ is a two-player game $\GGG^U$ defined as follows: Player I starts with some $s_0 \in U$ of length $1$, then each player has to play an $s_{n+1} \vartriangleleft s_n$. The last player who can move wins the game. A winning strategy $\sigma$ for Player I in $\GGG^U$ is a subtree $\sigma \subseteq U$ such that for all $s\in \sigma$ of even length $|s|$ there is exactly one $t\in\sigma$ with $t \vartriangleleft s$, and for every $s \in \sigma$ of odd length, every $t \vartriangleleft s$ is in $\sigma$. A winning strategy for Player II is defined interchanging odd with even in the above statement. A game $\GGG^U$ is determined if one of the two players has a winning strategy.

	We recall that there is a correspondence between games $\GGG^U$ on a well-founded tree $U$ and games on a pruned tree (as defined in \cite{kechris:descriptive}) whose winning condition is a clopen set. This justifies our terminology.

	In the following we will be interested in theories $T \supseteq \NBG + \ADF$, where $\ADF$ is 
	the following axiom of determinacy for clopen class games:

	\begin{definition}[$\ADF$]
		$\GGG^U$ is determined for any well-founded class tree $U$.
	\end{definition}
	
	Games $\GGG^U$ on well-founded set trees $U\subseteq V_\delta$ are determined in $\ZFC$ (see \cite[Thm. 20.1]{kechris:descriptive}) and the corresponding strategies $\sigma \subseteq U$ are elements of $V_{\delta+1}$. Thus the theory $\NBG + \ADF$ holds in any $V_{\delta+1}$ with $\delta$ inaccessible, and we can apply the results of this section to this theory. A finer upper bound for $\NBG + \ADF$ is given by the next proposition (we thank Alessandro Andretta for pointing this fact to us).

	\begin{proposition} \label{prop:mkadf}
		The Morse-Kelley theory $\MK$ (with the axiom of global choice) implies $\ADF$.
	\end{proposition}
	\begin{proof}
		Since the recursion theorem on well-founded class trees holds in $\MK$ (see \cite[Prop. 2]{antos:class_forcing}), we can follow the classical $\ZFC$ proof of determinacy for clopen games $\GGG^U$ on well-founded set trees $U$.

		For any $s \in U$ the next moving player is I if $\vp{s} \equiv 0 \pmod{2}$ and II otherwise. Define recursively a (class) map $w : U \to 2$ assigning  (coherently) to every position $s$ in $\GGG^U$ a ``supposedly winning'' player $w(s)$ (I if $\vp{s} \equiv 0 \pmod{2}$ and II otherwise). Precisely, we let $w(s) \equiv \vp{s}  \pmod{2}$ (that is, the moving player is winning) iff there exists a $t \vartriangleleft s$ such that $w(t) \equiv \vp{s}  \pmod{2}$ (that is, the moving player can win by playing $t$).

		Then we can use the map $w$ and the axiom of global choice to define a winning strategy $\sigma$ for Player I if $w(\emptyset) = 0$, and for Player II if $w(\emptyset) = 1$. Precisely, define $U_w = \bp{s \in U: ~ w(s) = w(\emptyset)}$ and $s^+ = \min_{\lessdot}\bp{t \in U_w: ~ t \vartriangleleft s}$ where $\lessdot$ is a well-order on $V$. Then,
		\[
		\sigma = \bp{s \in U_w: ~ \forall m < \vp{s} ~~ \cp{m \equiv w(\emptyset) \pmod{2}} \rightarrow \cp{s \res (m+1) = (s \res m)^+}}
		\]
		is the desired strategy.
	\end{proof}
	
	Theories $\NBG+\ADF$ and $\MK$ are well-suited for proving statements to be $\Delta^1_1(T)$, as shown in the following.
	
	\begin{proposition} \label{prop:winningdelta}
		Let $U$ be a first-order definable well-founded class tree, and let $\phi$ be asserting that ``player $p$ wins the game $\GGG^U$''. Then $\phi$ is $\Delta^1_1(\NBG + \ADF)$ and canonical $\Delta^1_1(\MK)$.
	\end{proposition}
	\begin{proof}
		By $\ADF$, $\phi$ can be expressed both as \emph{``there is a strategy $\sigma$ winning for $p$''} and as \emph{``all strategies $\sigma$ are not winning for $1-p$''}. Since being a winning strategy for a first-order definable $U$ is first-order expressible, the first part follows.
		
		For the second part, $\phi$ is canonically $\Delta^1_1(\MK)$ as witnessed 
		by the unique class $w : U \to 2$ defined in Proposition \ref{prop:mkadf}, which exists assuming $\MK$.
	\end{proof}
	
	In the remainder of this paper we will be working in $T = \MK$ and use $\Delta^1_1$ as a shorthand for $\Delta^1_1(\MK)$, while pointing out some passages where $\ADF$ is essentially used. 
	It is possible that a theory $T$ weaker than $\MK$, such as $\NBG+\ADF$ or $\NBG +\textit{there exists a truth predicate}$, 
	would suffice to carry out all the arguments at hand. However, it is not clear to us at the moment 
	whether such theories $T$ are preserved by set-sized forcings (as it is the case for $\MK$),
	and whether they allow to formulate each 
	$\Delta^1_1(\MK)$-property we will introduce in
	the remainder of this paper as a $\Delta^1_1(T)$-property.

	\subsection{Algebraic forcing iterations}

	We present iterations following \cite{a:viale:semiproper_iterations}, which expands on the work of Donder and 
	Fuchs on revised countable support iterations~\cite{fuchs:donder_rcs}. We will need the material in this 
	section to prove the consistency of the iterated resurrection axioms.
	
	Let $(\PP,\leq_\PP),(\QQ,\leq_\QQ)$ be partial orders. We recall that $\PP$ is a suborder of $\QQ$ if $\PP\subseteq \QQ$ and the inclusion map preserves the order and the incompatibility relation. $\PP$ is a complete suborder of $\QQ$ if any maximal antichain in $(\PP,\leq_\PP)$ remains such in $(\QQ,\leq_\QQ)$.
	We feel free to confuse a partial order $\PP$ with its Boolean completion $\RO(\PP)$ and a Boolean algebra $\BB$ with the partial order $\BB^+$ given by its positive elements.
	
	\begin{definition}
		Let $\BB$, $\CC$ be complete Boolean algebras, $i: \BB \to \CC$ is a \emph{complete homomorphism} iff it is an homomorphism that preserves arbitrary suprema. We say that $i$ is a \emph{regular embedding} iff it is an injective complete homomorphism of Boolean algebras.
	\end{definition}

	Complete homomorphisms on complete Boolean algebras induce natural \linebreak $\Delta_1$-elementary maps between the corresponding Boolean-valued models.
	
	\begin{proposition}[{\cite[Prop. 2.11]{a:viale:semiproper_iterations}}] \label{prop:ihat}
		Let $i:\BB\to\CC$ be a complete homomorphism, and define by recursion $\hat{\imath}: V^\BB \to V^\CC$ as
		\[
		\hat{\imath}(\dot{x}) = \bp{ \ap{\hat{\imath}(\dot{y}),i(\dot{x}(\dot{y}))}: ~ \dot{y} \in \dom(\dot{x}) }
		\]
		Then the map $\hat{\imath}$ is $\Delta_1$-elementary, i.e., for every $\Delta_1$-property\footnote{More
		 precisely, provably $\Delta_1$ in a theory $T$ which holds in $V^\BB$ and $V^\CC$.} $\phi$,
		\[
		i\cp{ \Qp{ \phi(\dot{x}_1,\ldots,\dot{x}_n) }_\BB } = \Qp{ \phi( \hat{\imath}(\dot{x}_1),\ldots,\hat{\imath}(\dot{x}_n) ) }_\CC
		\]
	\end{proposition}
	
	We are now ready to give an algebraic definition of forcing iteration.
	
	\begin{definition}
		Let $i: \BB \to \CC$ be a regular embedding, the \emph{retraction} associated to $i$ is the map
		\[
		\begin{array}{llll}
			\pi_i :& \CC &\to& \BB \\
			& p &\mapsto& \bigwedge \bp{q \in \BB: ~ i(q) \geq p}.
		\end{array}
		\]
	\end{definition}
	
	\begin{definition}
		$\FFF=\{i_{\alpha \beta}:\BB_\alpha\to \BB_\beta:\alpha \leq \beta < \lambda\}$ is a \emph{complete iteration system} of complete Boolean algebras iff for all $\alpha \leq \beta \leq \xi < \lambda$:
		\begin{enumerate}
			\item $\BB_\alpha$ is a complete Boolean algebra and $i_{\alpha \alpha}$ is the identity on it;
			\item $i_{\alpha \beta}$ is a regular embedding with associated retraction $\pi_{\alpha \beta}$;
			\item $i_{\beta \xi}\circ i_{\alpha \beta}=i_{\alpha \xi}$.
	 \end{enumerate}
		If $\xi < \lambda$, we define $\FFF \res \xi = \bp{i_{\alpha \beta}: \alpha \leq \beta < \xi}$.
	\end{definition}
	
	\begin{definition}
	Let $\FFF$ be a complete iteration system of length $\lambda$. Then:
	
	\begin{itemize}
	\item
	 The \emph{inverse limit} of the iteration is
		\[
		\invlim \FFF = \bp{ s \in \prod_{\alpha < \lambda} \BB_\alpha : ~ \forall \alpha \forall \beta > \alpha ~ \pi_{\alpha \beta}(s(\beta)) = s(\alpha) }
		\]
		and its elements are called \emph{threads}. 
	\item
		The \emph{direct limit} is
		\[
		\dirlim \FFF = \bp{ s \in \invlim \FFF : ~ \exists \alpha \forall \beta > \alpha ~ s(\beta) = i_{\alpha\beta}(s(\alpha)) }
		\]
		and its elements are called \emph{constant threads}.
	\item
		The \emph{revised countable support limit} is
		\[
		\rcslim \FFF = \bp{ s \in \invlim \FFF : ~ s \in \dirlim \FFF \vee \exists \alpha ~ s(\alpha) \Vdash_{\BB_\alpha} \cof(\check{\lambda}) = \check{\omega} }.
		\]
	\end{itemize}	
	\end{definition}
	
	\begin{definition}
		Let $\FFF = \bp{i_{\alpha \beta} : ~ \alpha \leq \beta < \lambda}$ be an iteration system. We say that $\FFF$ is a \emph{${<}\kappa$-support iteration} iff $\BB_\alpha = \RO(\dirlim \FFF \res \alpha)$ whenever $\cof(\alpha) \geq \kappa$, and $\BB_\alpha = \RO(\invlim \FFF \res \alpha)$ otherwise. We say that $\FFF$ is a \emph{revised countable support} iteration iff $\BB_\alpha = \RO(\rcslim \FFF \res \alpha)$ for all $\alpha < \lambda$.
	\end{definition}
	
	We recall that the direct limit of an iteration system inherits the structure of 
	a Boolean algebra through pointwise operations.
	
	\begin{theorem}[Baumgartner, {\cite[Prop. 3.13]{a:viale:semiproper_iterations}}] \label{iBaumgartner}
		Let $\FFF$ be an iteration system such that $\BB_\alpha$ is ${<}\lambda$-cc for all $\alpha$ and $S = \bp{\alpha: ~ \BB_\alpha \cong \RO(\dirlim \FFF \res \alpha)}$ is stationary. Then $\dirlim \FFF$ is ${<}\lambda$-cc and $\invlim \FFF = \dirlim \FFF$ is a complete Boolean algebra.
	\end{theorem}
	
	We will need the following formulation of two-step iterations and generic quotients.
	
	\begin{definition}[two-step iteration] \label{def:two_step_iter}
		Let $\BB$ be a complete Boolean algebra, and $\dot{\CC}$ be a $\BB$-name for a complete Boolean algebra. We denote by $\BB \ast \dot{\CC}$ the Boolean algebra defined in $V$ whose elements are the equivalence classes of $\BB$-names for elements of $\dot{\CC}$ (i.e., $\dot{p}\in V^{\BB}$ such that $\Qp{\dot{p}\in\dot{\CC}}_{\BB}=\1$) modulo the equivalence relation:
		\[
		\dot{p} \approx \dot{q} ~ \Leftrightarrow ~ \Qp{\dot{p} = \dot{q}}_\BB = \1,
		\]
		with the following operations:
		\begin{align*}
			[\dot{p}] \vee_{\BB \ast \dot{\CC}} [\dot{q}] = [\dot{r}] &\iff \Qp{\dot{r} = \dot{p} \vee_{\dot{\CC}} \dot{q}}_\BB =\1; \\
			\qquad \neg_{\BB \ast \dot{\CC}}[\dot{p}]=[\dot{r}] &\iff \Qp{\dot{r}=\neg_{\dot{\CC}}\dot{p}}_\BB = \1.
		\end{align*}
	\end{definition}
	
	\begin{proposition}[{\cite[Prop. 4.6]{a:viale:semiproper_iterations}}]\label{qQuotientCba}
		Let $i: \BB \to \CC$ be a regular embedding of complete Boolean algebras and $G$ be a $V$-generic filter for $\BB$. Then $\CC /_G$, defined (with abuse of notation) as the quotient of $\CC$ by the ideal which is dual to the filter generated by $i[G]$, is a complete Boolean algebra in $V[G]$.
	\end{proposition}
	
	\begin{definition}
		Let $\Gamma$ be any definable class of complete Boolean algebras closed under two step iterations and let $\BB, \CC$ be complete Boolean algebras.

		We say that $\BB \leq_\Gamma \CC$ iff there is a complete homomorphism $i: \CC \to \BB$ such that the quotient algebra $\BB /_{\dot{G}_\CC}$ is in $\Gamma$ with Boolean value $\1_{\CC}$. Notice that we do not require either $\BB$ or $\CC$ to be in $\Gamma$.

		We say that $\BB \leq^\ast_\Gamma \CC$ iff there is a complete \emph{injective} homomorphism with the same properties as above.
	\end{definition}

	\subsection{Weakly iterable forcing classes}

	We are now ready to introduce the definition of \emph{weakly iterable} class of forcing notions. Given a definable class $\Gamma$ of forcing notions, let $\Gamma^{\lim}$ denote the (definable) class of complete iteration systems $\FFF = \bp{i_{\alpha \beta}:\BB_\alpha\to \BB_\beta: ~ \alpha \leq \beta < \lambda}$ such that $i_{\alpha \beta}$ witnesses that $\BB_\beta \leq_\Gamma \BB_\alpha$ for all $\alpha \leq \beta < \lambda$.
	
	\begin{definition}
		Let $T$ be a theory extending $\NBG$, $\Gamma$ be a definable class of complete Boolean algebras in $T$, 
		$\Sigma : \Gamma^{\lim} \to \Gamma^{\lim}$ be a definable class function in $T$, 
		$\gamma$ be a definable cardinal in $T$.
		We say that an iteration system $\FFF=\{\BB_\eta:\eta<\alpha\} \in \Gamma^{\lim}$ of length $\alpha$ 
		\emph{follows} $\Sigma$ if and only if for all $\beta < \alpha$ 
		even\footnote{We remark that every limit ordinal is even.}, 
		$\FFF \res (\beta+1) = \Sigma(\FFF \res \beta)$.
	
		We say that $\Sigma$ is a \emph{weak iteration strategy} for $\Gamma$ if and only if we can prove in 
		$T$ that for every $\FFF=\{\BB_\eta:\eta<\alpha\}$ of length $\alpha$ which follows $\Sigma$, 
		$\Sigma(\FFF)$ has length $\alpha+1$ and $\FFF = \Sigma(\FFF) \res \alpha$.
	
		 We say that $\Sigma$ is a \emph{$\gamma$-weak iteration strategy} for $\Gamma$ if in addition 
		 $\Sigma(\FFF) = \dirlim \FFF$ whenever $\FFF=\{\BB_\eta:\eta<\alpha\}$ and $\cof(\alpha) = \gamma$ 
		 or $\cof(\alpha) =\alpha> \gamma,\vp{\BB}$ for all $\BB$ in $\FFF$.
	\end{definition}
	
	\begin{definition}
		Let $\BBB$ be a collection of complete Boolean algebras. We denote as $\prod \BBB$ the \emph{lottery sum} of the algebras in $\BBB$, defined as the Boolean algebra obtained by the cartesian product of the respective Boolean algebras with pointwise operations.
	\end{definition}
	
	The name \emph{lottery sum} is justified by the intuition that forcing with $\prod \BBB$ corresponds with forcing with a ``random'' algebra in $\BBB$. In fact, since the set of $p \in \prod \BBB$ that are $\1$ in one component and $\0$ in all the others form a maximal antichain, every $V$-generic filter $G$ for $\prod \BBB$ concentrates only on a specific $\BB \in \BBB$ (determined by the generic filter).
	
	\begin{definition} \label{def:weak_iterable}
		Let $T$ be a theory extending $\NBG$ by a recursive set of axioms, $\Gamma$ be a definable class of complete Boolean algebras, $\Sigma : \Gamma^{\lim} \to \Gamma^{\lim}$ be a definable class function, $\gamma$ be a definable cardinal.
	
		We say that $\Gamma$ is \emph{$\gamma$-weakly iterable} through $\Sigma$ in $T$ iff we can prove in $T$ that:
		\begin{itemize}
			\item $\Gamma$ is closed under two-step iterations and set-sized lottery sums;
			\item $\Sigma$ is a $\gamma$-weak iteration strategy for $\Gamma$;
			\item $\ap{\Gamma, \Sigma}$ as computed in $V_{\kappa+1}$ is equal to $\ap{\Gamma \cap V_\kappa, \Sigma \cap V_\kappa}$ whenever $\kappa$ is inaccessible and $V_{\kappa+1} \models T$.%
			\footnote{Since $\kappa$ is inaccessible, this statement is equivalent to $V_{\kappa+1} \models T \setminus \NBG$ which is recursive.}
		\end{itemize}
		We say that $\Gamma$ is weakly iterable iff it is $\gamma$-weakly iterable through 
		$\Sigma$ for some $\gamma,\Sigma$.
	\end{definition}
	
	Intuitively, $\Gamma$ is weakly iterable iff there is a sufficiently nice strategy for choosing limits in $\Gamma$
 	for iterations of indefinite length in $\Gamma^{\lim}$.
	We remark that the latter definition (for a $T \supseteq \NBG$) is not related to a specific model $V$ of $T$, and requires that the above properties are provable in $T$, and hence hold for \emph{every} $T$-model $M$: for example, if $T = \MK$ they must hold in every $V_{\kappa+1}$ where $\kappa$ is inaccessible. We feel free to omit the reference to $T$ when clear from the context, and in particular when $T = \MK$.
	
	Many notable classes $\Gamma$ are $\omega_i$-weakly iterable for some $i=0,1$:
	\begin{itemize}
		\item $\Omega$ is $\omega$-weakly iterable using a strategy $\Sigma$ which takes 
		finite support limits at limit stages and is the identity elsewhere.
		
		\item axiom-$A$, proper are $\omega_1$-weakly iterable using a strategy $\Sigma$
		which takes countable support limits at limit stages and is the identity elsewhere \cite{baumgartner:iterated_forcing}.
		
		\item locally ${<}\kappa^+$-cc\footnote{$\BB$ is locally ${<}\kappa$-cc if it is the lottery sum of ${<}\kappa$-cc complete Boolean algebras.}
		and ${<}\kappa$-closed are $\kappa$-weakly iterable using a strategy $\Sigma$ which takes ${<}\kappa$-sized support limits at limit stages and is the identity elsewhere \cite{baumgartner:iterated_forcing}.
		
		\item $\SP$ is $\omega_1$-weakly iterable using a strategy $\Sigma$
		which takes revised countable support limits, and chooses 
		$\BB_{\alpha} = \BB_{\alpha-1} \ast \Coll(\omega_1, \vp{\BB_{\alpha-1}})$ 
		\cite[Thm. 7.11]{a:viale:semiproper_iterations} at even successor stages $\alpha$.
		
		\item $\SSP$ is $\omega_1$-weakly iterable assuming the existence of a proper class of 
		supercompact cardinals, using a strategy $\Sigma$
		which takes revised countable support limits
		and chooses $\BB_{\alpha} = \BB_{\alpha-1} \ast \dot\CC$, where $\dot{\CC}$ forces $\SP = \SSP$ and collapses $\vp{\BB_{\alpha-1}}$ to have size $\omega_1$ \cite[Thm. 3.4.17]{a:thesis}.%
		\footnote{Such a $\BB$-name for a forcing $\dot\CC$ exists, an example being the lottery iteration (Def. \ref{def:lotteryiteration}) of semiproper forcings of length a supercompact $\kappa> \vp{\BB_{\alpha-1}}$. To ensure the definability of $\Sigma$, this choice can be made canonical by considering the lottery sum of all the Boolean algebras of minimum rank satisfying the required property.}
	\end{itemize}
	
	The definition of weak iterability for a definable class of forcing notions $\Gamma$ provides the 
	right conditions to carry out the \emph{lottery iteration} $\PP^{\Gamma,f}_\kappa$ with respect to a partial function 
	$f: \kappa \to \kappa$ where $\kappa$ is an inaccessible cardinal. The lottery iteration has been studied 
	extensively by 
	Hamkins~\cite{hamkins:lottery_preparation} and is one of the main tools to obtain the consistency of forcing axioms. 
	We will employ these type of iterations in Section \ref{sec:consistency}.
	
	\begin{definition} \label{def:lotteryiteration}
		Let $\Gamma$ be $\gamma$-weakly iterable through $\Sigma$ and $f: \kappa \to \kappa$ be a partial function. 
		Define $\FFF_\xi = \bp{\PP^{\Gamma, f}_\alpha : \alpha < \xi}$ by recursion on $\xi \leq \kappa+1$ as:
		\begin{enumerate}
			\item \label{def:lottery-0} $\FFF_0 = \emptyset$ is the empty iteration system;
			\item \label{def:lottery-1} $\FFF_{\xi+1} = \Sigma(\FFF_\xi)$ if $\xi$ is even;
			\item \label{def:lottery-2} $\FFF_{\xi+2}$ has $\PP^{\Gamma, f}_{\xi+1} = \PP^{\Gamma, f}_\xi$ if $\xi+1$ is odd and $f(\xi)$ is undefined;
			\item \label{def:lottery-3} $\FFF_{\xi+2}$ has $\PP^{\Gamma, f}_{\xi+1} = \PP^{\Gamma, f}_\xi \ast \dot{\CC}$ otherwise, where $\dot{\CC}$ is a $\PP^{\Gamma, f}_\xi$-name for the lottery sum (as computed in $V^{\PP^{\Gamma, f}_\xi}$) of all complete Boolean algebras in $\Gamma$ of rank less than $f(\xi)$, i.e., a $\PP^{\Gamma, f}_\xi$-name for $\prod \cp{\Gamma \cap V_{f(\xi)}}$.
		\end{enumerate}
		We say that $\PP^{\Gamma, f}_\kappa$ is the lottery iteration of $\Gamma$ relative to $f$.
	\end{definition}
	
	\begin{proposition} \label{prop:lottery_iteration}
		Let $T$ be a theory extending $\NBG$ by a recursive set of axioms, $\Gamma$ be $\gamma$-weakly iterable through $\Sigma$, $f: \kappa \to \kappa$ be a partial function with $\kappa > \gamma$ inaccessible cardinal such that $V_{\kappa+1} \models T$. Then:
	\begin{enumerate}
		\item \label{iter_goal-1} $\PP^{\Gamma,f}_\kappa$ exists and is in $\Gamma$;
		\item \label{iter_goal-2} $\PP^{\Gamma,f}_\kappa$ is ${<}\kappa$-cc and for all $\alpha < \kappa$, $\1 \Vdash_{\PP^{\Gamma,f}_\kappa} 2^{\check{\alpha}} \leq \check{\kappa}$;
		\item \label{iter_goal-3} $\PP^{\Gamma,f}_\kappa$ is definable in $V_{\kappa+1}$ using the class parameter $f$;
		\item \label{iter_goal-4} Let $g: \lambda \to \lambda$ with $\lambda$ inaccessible be such that $f = g \res \kappa$, $V_\lambda \models T$. Then $\PP^{\Gamma,g}_\lambda$ absorbs every forcing in $\Gamma \cap V_{g(\kappa)}$ as computed in $V^{\PP^{\Gamma,f}_\kappa}$. That is, for every $\dot{\BB}$ in $(\Gamma \cap V_{g(\kappa)})^{V^{\PP^{\Gamma,f}_\kappa}}$, there is a condition $p \in \PP^{\Gamma,g}_\lambda$ such that $\PP^{\Gamma,g}_\lambda \res p \leq_\Gamma \PP^{\Gamma,f}_\kappa \ast \dot{\BB}$.
	\end{enumerate}
	\end{proposition}
	\begin{proof}
	\emph{}
	
		\begin{enumerate}
			\item
			Follows from $\Sigma$ being a weak iteration strategy for $\alpha < \kappa$ even, and from $\Gamma$ closed under two-step iterations and lottery sums for $\alpha$ odd.

			\item 
			Since $\Sigma \cap V_\kappa$ is $\Sigma$ as computed in $V_{\kappa+1}$, we can prove by induction on $\alpha < \kappa$ that $\vp{\PP^{\Gamma,f}_\alpha} < \kappa$, hence $\PP^{\Gamma,f}_\alpha$ is ${<}\kappa$-cc for all $\alpha<\kappa$. Furthermore, by definition of $\gamma$-weak iteration strategy the set of ordinals:
			\[
			S = \bp{\alpha < \kappa: ~ \PP^{\Gamma,f}_\alpha = \dirlim \cp{\FFF_\kappa \res \alpha}} \supseteq \bp{\alpha < \kappa: ~ \cof(\alpha) = \gamma}
			\]
			is stationary in $\kappa$. It follows by Baumgartner's Theorem \ref{iBaumgartner} that $\dirlim \FFF_\kappa = \invlim \FFF_\kappa$ is ${<}\kappa$-cc. Since $\Sigma$ is a $\gamma$-weak iteration strategy and $\kappa = \cof(\kappa) > \gamma, \vp{\PP^{\Gamma,f}_\alpha}$ for all $\alpha < \kappa$, $\PP^{\Gamma, f}_\kappa = \Sigma(\FFF_\kappa) = \dirlim \FFF_\kappa$ concluding the proof of the first part of the statement.
			
			For the second part, given $\alpha < \kappa$ let $\dot{x}$ be a $\PP^{\Gamma, f}_\kappa$-name for a subset of $\alpha$. Then $\dot{x}$ is decided by $\alpha < \kappa$ antichains of size ${<}\kappa$, hence $\dot{x} = \hat{\imath}_\beta(\dot{y})$%
			\footnote{We recall that $\hat\imath_\beta$ is the natural embedding from $V^{\PP^{\Gamma, f}_\beta}$ to $V^{\PP^{\Gamma, f}_\kappa}$ (see Proposition \ref{prop:ihat}).}
			for some $\dot{y} \in V^{\PP^{\Gamma, f}_\beta}$, $\beta < \kappa$. 
			Since $\vp{\PP^{\Gamma, f}_\beta} < \kappa$ and $\kappa$ is inaccessible, there are less than $\kappa$-many names for subsets of $\alpha$ in $V^{\PP^{\Gamma, f}_\beta}$. Thus there are at most $\kappa$-many names for subsets of $\alpha$ in $V^{\PP^{\Gamma, f}_\kappa}$.
			
			\item
			Straigthforward, given that $\ap{\Gamma \cap V_\kappa, \Sigma \cap V_\kappa}$ is $\ap{\Gamma, \Sigma}$ as computed in $V_{\kappa+1}$.
	
			\item
			Let $g : \lambda \to \lambda$ and $\lambda$ inaccessible be such that $f = g \res \kappa$ and $V_{\lambda+1} \models T$. Since $\ap{\Gamma \cap V_\kappa, \Sigma \cap V_\kappa}$ is $\ap{\Gamma, \Sigma}$ as computed in $V_{\kappa+1}$ and the same holds for $\lambda$, letting $\FFF=\{\PP^{\Gamma,g}_\alpha:\alpha<\lambda\}$, we have that 
			\[
			\FFF\res\kappa=\{\PP^{\Gamma,g}_\alpha:\alpha<\kappa\} = \{\PP^{\Gamma,f}_\alpha:\alpha<\kappa\}.
			\]
			Hence $\PP^{\Gamma,f}_\kappa=\PP^{\Gamma,g}_\kappa$.
			Furthermore any $\dot{\BB}$ in $(\Gamma \cap V_{g(\kappa)})^{V^{\PP^{\Gamma,f}_\kappa}}$ is forced by $\PP^{\Gamma,f}_\kappa$ to be the restriction of $\dot{\CC} = \prod \cp{\Gamma \cap V_{g(\kappa)}}^{V^{\PP^{\Gamma,f}_\kappa}}$ to a suitable condition $\dot{q}$ applying Definition \ref{def:lotteryiteration}.(\ref{def:lottery-3}) in $V^{\PP^{\Gamma,f}_\kappa}$ ($\dot{q}$ is forced by $\PP^{\Gamma,f}_\kappa$ to be $\1$ in the component of $\dot{\CC}$ corresponding to $\dot{\BB}$ and $\0$ in all the other components of $\dot{\CC}$).
			Hence we can find $p \in \PP^{\Gamma,g}_\lambda$ such that $\PP^{\Gamma,g}_{\kappa+1}\res (p\res{\kappa+1})\cong 
			\PP^{\Gamma,f}_\kappa\ast\dot{\BB}$. Therefore,
			\[
			\PP^{\Gamma,f}_\kappa \ast \dot{\BB} ~\cong~ 
			\PP^{\Gamma,g}_{\kappa+1} \res (p\res{\kappa+1}) 
			~\geq_\Gamma~ \PP^{\Gamma,g}_\lambda  \res p. \qedhere
			\]
		\end{enumerate}
	\end{proof}

\subsection{Forcing axioms as density properties} \label{sec:density}

	The resurrection axiom, introduced by Hamkins and Johnstone in \cite{hamkins:resurrection_uplifting}, 
	can be naturally phrased as a density property for the class partial order $(\Gamma,\leq_\Gamma)$.

	\begin{definition}[Hamkins, Johnstone~\cite{hamkins:resurrection_uplifting}] \label{def:hj_ra}
		The \emph{resurrection axiom} $\RA(\Gamma)$ is the assertion that the class
		\[
		\bp{\BB \in \Gamma: ~~ H_\cc \prec H_\cc^\BB}
		\]
		is dense in $\cp{\Gamma, \leq_\Gamma}$. 

		The \emph{weak resurrection axiom} $\wRA(\Gamma)$ is the assertion that for all $\BB \in \Gamma$, there exists a $\CC \leq_{\ALL} \BB$ such that $H_\cc \prec H_\cc^\CC$.
	\end{definition}

	We can also reformulate in a similar way many of the common forcing axioms. We first recall the following standard definitions.

	\begin{definition}
		The bounded forcing axiom $\BFA_\kappa(\Gamma)$ holds if for all complete Boolean algebras 
		$\BB\in\Gamma$ and all families 
		$\{D_\alpha:\alpha<\kappa\}$ of predense subsets of $\BB$ of size at most $\kappa$, there is a filter 
		$G\subset\BB$ meeting all these sets.

		The forcing axiom $\FA_\kappa(\Gamma)$ holds if for all 
		$\BB\in\Gamma$ and all families $\{D_\alpha:\alpha<\kappa\}$ of predense subsets of 
		$\BB$, there is a filter $G\subset\BB$ meeting all these sets.

		The forcing axiom $\FA^{++}_{\omega_1}(\Gamma)$ holds if for all $\BB\in\Gamma$ and all families 
		$\{D_\alpha:\alpha<\omega_1\}$ of predense subsets of $\BB$ and all families 
		$\{\dot{S}_\alpha:\alpha<\omega_1\}$ of $\BB$-names for stationary subsets of $\omega_1$, 
		there is a filter $G\subset\BB$ meeting all these dense sets and evaluating each $\dot{S}_\alpha$ 
		as a stationary subset of $\omega_1$. 
		
		$\FA^{++}_{\omega_1}(\SSP)$ is better known under the name $\MM^{++}$. $\BFA(\Gamma)$ is a shorthand for $\BFA_{\omega_1}(\Gamma)$ and $\FA(\Gamma)$ a shorthand for $\FA_{\omega_1}(\Gamma)$.
	\end{definition}

	\begin{theorem}[Bagaria, \cite{bagaria:bfa_absoluteness}]
		$\BFA(\Gamma)$ is equivalent to the assertion that the class
		\[
		\Gamma_0 = \bp{\BB \in \Gamma: ~~ H_{\omega_2} \prec_1 V^\BB}\footnotemark
		\]
		\footnotetext{We recall that this notation is introduced in Section \ref{ssec:notation}.}
		is dense in $\cp{\Gamma, \leq_{\ALL}}$.
	\end{theorem}
	
	We remark that the latter assertion is actually equivalent to requiring the class $\Gamma_0$ to coincide with the whole $\Gamma$ (since $\Sigma_1$-formulae are always upwards absolute).

	Under suitable large cardinal assumptions the unbounded versions of the forcing axioms can also be reformulated as density properties, but only for $\Gamma = \SSP$ (at least to our knowledge).

	\begin{theorem}[Woodin, \cite{larson:stationary_tower}]
		Assume there are class-many Woodin cardinals. 
		Then $\MM$ (i.e., $\FA(\SSP)$) is equivalent to the assertion that the class
		\[
		\bp{\BB \in \SSP: ~~ \BB \text{ is a presaturated tower forcing}}\footnotemark
		\]
		\footnotetext{We refer the reader to~\cite[Def. 5.8]{viale:category_forcing} for a definition of presaturated and strongly presaturated tower forcings.}
		is dense in $\cp{\SSP, \leq_{\ALL}}$.
	\end{theorem}

	\begin{theorem}[Viale, \cite{viale:category_forcing}]
		Assume there are class-many Woodin cardinals. Then $\MM^{++}$ is equivalent to the assertion that the above class is dense in $\cp{\SSP, \leq_{\SSP}}$.
	\end{theorem}

	 In this paper we also refer to the following strengthening of $\MM^{++}$, which is defined by a density property of the class $\SSP$ as follows.

	\begin{definition}[Viale, \cite{viale:category_forcing}]
		$\MM^{+++}$ is the assertion that the class
		\begin{equation*}
			\bp{\BB \in \SSP: ~~ \BB \text{ is  a strongly presaturated tower}}
		\end{equation*}
		is dense in $\cp{\SSP, \leq_{\SSP}}$.
	\end{definition}

	\section{Iterated resurrection and absoluteness} \label{sec:absoluteness}

	We now introduce the iterated resurrection axioms $\RA_\alpha(\Gamma)$ with $\Gamma$ a definable class of complete Boolean algebras, and prove that the iterated resurrection axiom $\RA_\omega(\Gamma)$ gives generic absoluteness for the first-order theory (with parameters) of $H_{\gamma^+}$ for a certain cardinal $\gamma=\gamma_\Gamma$ which is computed in terms of the combinatorial properties of $\Gamma$. In particular, we aim to choose $\gamma$ as large as possible while still being able to consistently prove the generic absoluteness of the theory of $H_{\gamma^+}$ with respect to $\Gamma$, at the same time maintaining that the axioms $\RA_\alpha(\Gamma_0),\RA_\alpha(\Gamma_1)$ with $\gamma_{\Gamma_0}\neq\gamma_{\Gamma_1}$ can be mutually compatible, if each $\Gamma_i$ is chosen properly.
	We are inspired by the resurrection axioms introduced by Hamkins and Johnstone in \cite{hamkins:resurrection_uplifting}, which are formulated in similar terms for $H_{\cc}$ and $\Gamma$ among the usual classes for which forcing axioms have been extensively studied.

	\begin{definition}
		Let $\Gamma$ be a definable  class of complete Boolean algebras closed under two step iterations.		
		The \emph{cardinal preservation degree} $\cpd(\Gamma)$ of $\Gamma$ is the maximum cardinal $\kappa$ such that every $\BB \in \Gamma$ forces that every cardinal $\nu \leq \kappa$ is still a cardinal in $V^\BB$. If all cardinals are preserved by $\Gamma$, we say that $\cpd(\Gamma) = \infty$.

		The \emph{distributivity degree} $\dd(\Gamma)$ of $\Gamma$ is the maximum cardinal 
		$\kappa$ such that every $\BB \in \Gamma$ is ${<}\kappa$-distributive.
	\end{definition}

	We remark that the supremum of the cardinals preserved by $\Gamma$ is preserved by $\Gamma$, and the same holds for the property of being ${<}\kappa$ distributive. Furthermore, $\dd(\Gamma) \leq \cpd(\Gamma)$ and $\dd(\Gamma)\neq\infty$ whenever $\Gamma$ is non trivial (i.e., it contains a Boolean algebra that is not forcing equivalent to the trivial Boolean algebra).
	Moreover $\dd(\Gamma)=\cpd(\Gamma)$ whenever $\Gamma$ is closed under two steps iterations and contains the class of ${<}\cpd(\Gamma)$-closed posets.
	
	\begin{definition} \label{def:gamma_gamma}
		Let $\Gamma$ be a definable class of complete Boolean algebras. We let $\gamma = \gamma_\Gamma = \cpd(\Gamma)$.
	\end{definition}
	
	For example, $\gamma = \omega$ if $\Gamma $ is the class of all posets, while for axiom-$A$, proper, $\SP$, $\SSP$ we have that $\gamma = \omega_1$ and for ${<}\kappa-$closed we have that $\gamma = \kappa$.
	
	We aim to isolate for each cardinal $\gamma$ classes of forcings $\Delta_\gamma$ and 
	axioms $\AX(\Delta_\gamma)$ such that:
	\begin{enumerate}
		\item \label{aim:resurrection-1}
		$\gamma=\cpd(\Delta_\gamma)$ and assuming certain large cardinal axioms, the family of $\BB\in \Delta_\gamma$ which force $\AX(\Delta_\gamma)$ is dense in $(\Delta_\gamma,\leq_{\Delta_\gamma})$;
		\item \label{aim:resurrection-2}
		$\AX(\Delta_\gamma)$ gives generic absoluteness for the theory with parameters of $H_{\gamma^+}$ with respect to all forcings in $\Delta_\gamma$ which preserve $\AX(\Delta_\gamma)$;
		\item
		the axioms $\AX(\Delta_\gamma)$ are mutually compatible for the largest possible family of cardinals $\gamma$ simultaneously;
		\item
		the classes $\Delta_\gamma$ are the largest possible for which the axioms $\AX(\Delta_\gamma)$ can possibly be consistent.
	\end{enumerate}

	Towards this aim remark the following:
	\begin{itemize} 
		\item
		$\dd(\Gamma)$ is the least possible cardinal $\gamma$ such that $\AX(\Gamma)$ is a non-trivial axiom asserting generic absoluteness for the theory of $H_{\gamma^+}$ with parameters. In fact, $H_{\dd(\Gamma)}$ is never changed by forcings in $\Gamma$. 
		\item
		$\cpd(\Gamma)$ is the maximum possible cardinal $\gamma$ for which an axiom $\AX(\Gamma)$ as above can grant generic absoluteness with respect to $\Gamma$ for the theory of $H_{\gamma^+}$ with parameters. To see this, let $\Gamma$ be such that $\cpd(\Gamma) = \gamma $ and assume towards a contradiction that there is an axiom $\AX(\Gamma)$ yielding generic absoluteness with respect to $\Gamma$ for the theory with parameters of $H_{\lambda}$ with $\lambda>\gamma^+$.

		Assume that $\AX(\Gamma)$ holds in $V$. Since $\cpd(\Gamma) = \gamma$, there exists a $\BB \in \Gamma$ which collapses $\gamma^+$. Let $\CC \leq_\Gamma \BB$ be obtained by property (\ref{aim:resurrection-1}) above for
		$\Gamma=\Delta_\gamma$, so that $\AX(\Gamma)$ holds in $V^\CC$, and remark that 
		$\gamma^+$ cannot be a cardinal in $V^\CC$ as well. 
		Then $\gamma^+$ is a cardinal in $H_{\lambda}$ and not in $H_{\lambda}^\CC$, witnessing failure of generic absoluteness and contradicting property (\ref{aim:resurrection-2}) for $\AX(\Gamma)$.
	\end{itemize}

	In the remainder of this paper we will see that the axiom $\RA_\omega(\Gamma)$ satisfies the first two of the above requirements, and is consistent for a variety of forcing classes $\Gamma$
	which also provide natural examples for the last two requirements. 
	We will come back later on with philosophical considerations outlining why the last two requirements are also natural.
	In particular, we will prove the consistency of $\RA_\omega(\Gamma)$ for forcing classes which are definable in $\NBG$, weakly iterable (see Definition \ref{def:weak_iterable}) and satisfying the following property.

	\begin{definition}
		Let $\Gamma$ be a definable class of complete Boolean algebras. We say that $\Gamma$ is \emph{well behaved} iff it is closed under two-step iterations, $\gamma = \cpd(\Gamma)<\infty$, and for all $\kappa > \gamma$ there are densely many $\BB \in \Gamma$ collapsing $\kappa$ to $\gamma$.
	\end{definition}

	The above requirement ensures that the cardinal preservation degree of $\Gamma$ gives a uniform bound for the cardinals which are preserved by the forcings in $\Gamma$. This property is easily checked to hold for all $\Gamma$ containing the ${<}\cpd(\Gamma)$-closed forcings and closed under two steps iterations. Throughout the remainder of this paper we will focus on classes $\Gamma$ which are well behaved with $\gamma = \gamma_\Gamma=\cpd(\Gamma)$ given by Definition \ref{def:gamma_gamma}, and implicitly assume this property whenever needed. In order to prove the consistency of $\RA_\alpha(\Gamma)$, we will also need to assume that $\Gamma$ is weakly iterable (see Definition \ref{def:weak_iterable}).


\subsection{Resurrection games}

	Motivated by Hamkins and Johnstone's~\cite{hamkins:resurrection_uplifting}, as well as by Tsaprounis'~\cite{tsaprounis:on_resurrection}, we introduce the following new class games and corresponding forcing axioms.

		\begin{definition}
			Let $\Gamma$ be well-behaved with $\cpd(\Gamma)=\gamma$.
			The $\Gamma$-weak resurrection game $\GGG^{\wRA}$ is as follows. 
			Player I (\emph{Kill}) plays couples $(\alpha_n, \BB_n)$ where $\alpha_n$ is an 
			ordinal such that $\alpha_{n+1} < \alpha_n$ and $\BB_n$ is such that 
			$\BB_{n+1} \leq_\Gamma \CC_n$. Player II (\emph{Resurrect}) 
			plays Boolean algebras $\CC_n$ such that 
			$H_{\gamma^+}^{\CC_n} \prec H_{\gamma^+}^{\CC_{n+1}}$ and $\CC_n \leq_{\ALL} \BB_n$. 
			The last player who can move wins.

			The $\Gamma$-resurrection game $\GGG^{\RA}$ is the same game as $\GGG^{\wRA}$ with 
			the additional requirement for Player II (\emph{Resurrect}) to play so that 
			$\CC_n \leq_\Gamma \BB_n$ for all $n$.
		\end{definition}

		\begin{definition}
			The \emph{$\alpha$-weak resurrection axiom} $\wRA_\alpha(\Gamma)$ is the assertion that 
			Player II (\emph{Resurrect}) wins the $\Gamma$-weak resurrection game after 
			$\ap{(\alpha, \2), \2}$.\footnote{We recall that $\2$ denotes the two-valued Boolean algebra $\bp{\0, \1}$.}

			The \emph{$\alpha$-resurrection axiom} 
			$\RA_\alpha(\Gamma)$ is the assertion that Player II (\emph{Resurrect}) wins the 
			$\Gamma$-resurrection game after $\ap{(\alpha, \2), \2}$.

			We say that $\wRA_{\ON}(\Gamma)$ (respectively $\RA_{\ON}(\Gamma)$) holds iff the corresponding axioms hold for all $\alpha \in \ON$.
		\end{definition}

		Note that all these games are clopen class games, since Player I plays descending sequences of ordinals. Since we require $\Gamma$ to be first-order definable, the class tree of partial plays is first-order definable as well and the corresponding axioms asserting the existence of a class winning strategy for Player II in the relevant game (equivalently, that all strategies in the relevant game are not winning for Player I) are $\Delta^1_1$-statements\footnote{We recall that we use $\Delta^1_1$ as a shorthand for $\Delta^1_1(\MK)$, as stated in Section \ref{sec:class_games}.} as showed in Proposition \ref{prop:winningdelta}.

		A posteriori the axioms $\RA_\alpha(\Gamma)$ can also be formulated by a recursive process using the axioms $\RA_\beta(\Gamma)$ for $\beta < \alpha$, as we will see in the next proposition. However, this type of formulation cannot be directly used as a definition in $\LLL^2$. We will come back to this delicate point with more details after the proof of the next proposition.

		\begin{proposition}[$\ADF$] \label{prop:radef}
			$\wRA_\alpha(\Gamma)$ holds iff for all $\beta < \alpha$ and $\BB \in \Gamma$ there is a $\CC \leq_{\ALL} \BB$ such that $H_{\gamma^+} \prec H_{\gamma^+}^\CC$ and $V^\CC \models \wRA_\beta(\Gamma)$.

			Similarly, $\RA_\alpha(\Gamma)$ holds iff the same holds with $\CC \leq_\Gamma \BB$; 
			or equivalently for all $\beta < \alpha$ the class
			\[
			\bp{\BB \in \Gamma: ~~ H_{\gamma^+} \prec H^\BB_{\gamma^+} \wedge 
			V^\BB \models \RA_\beta(\Gamma)}
			\]
			is dense in $\cp{\Gamma, \leq_\Gamma}$.
		\end{proposition}
		\begin{proof}
		We divide the proof in two steps.
		\begin{description}
		\item[\textbf{Step 1}]
			First we prove that for any $\CC$, $\wRA_\beta(\Gamma)$ holds in $V^\CC$ iff Player II (\emph{Resurrect}) wins $\GGG^{\wRA}$ in $V$ after $\ap{(\beta, \CC), \CC}$.
			
			Let $G$ be $V$-generic for $\CC$ and $\DD \leq_{\ALL} \CC$ with witnessing map $i: \CC \to \DD$ ($i$ being the identity for $\DD=\CC$), and define $\DD_G := \DD /_{i[G]}$ in $V[G]$. Let $s$ in $\GGG^{\wRA}$ extend $\ap{(\beta, \CC), \CC}$ and define $s_G$ as the sequence obtained by substituting every Boolean algebra $\DD$ appearing in $s$ with $\DD_G$. 
			By $\ADF$, let $\sigma$ be a winning strategy for Player I or II in $\GGG^{\wRA}$ after $\ap{(\beta, \CC), \CC}$. Then $\sigma_G = \bp{s_G: ~ s \in \sigma}$ is a winning strategy for the same player in the corresponding game in $V[G]$ after $\ap{(\beta, \2), \2}$. It follows that $\GGG^{\wRA} \res \ap{(\beta, \2), \2}$ is determined in $V[G]$ 
			and Player I (\emph{Kill}) (resp. II (\emph{Resurrect})) wins
			$\GGG^{\wRA} \res \ap{(\beta, \2), \2}$  in $V[G]$ if and only if she
			(he) wins the corresponding game $\GGG^{\wRA} \res \ap{(\beta, \CC), \CC}$ in $V$.

			Furthermore, we observe that whenever there is a winning strategy $\sigma$ for Player I or II in $\GGG^{\wRA}$ and $s = \ap{(\alpha,\2),\2,(\beta,\BB),\CC}$ is in $\sigma$, we can define a winning strategy for the same player $\sigma_s = \bp{\ap{(\beta,\CC),\CC}^\smallfrown u: ~ s^\smallfrown u \in \sigma}$ in the game $\GGG^{\wRA} \res \ap{(\beta, \CC), \CC}$.

		\item[\textbf{Step 2}]
			We use the above results to prove the Lemma.

			First suppose that $\wRA_\alpha(\Gamma)$ holds, and fix $\beta < \alpha$, $\BB \in \Gamma$. Let $\sigma$ be a winning strategy for Player II (\emph{Resurrect}) in $\GGG^{\wRA}$ after $s_0 = \ap{(\alpha,\2),\2}$. Then $s_1 = s_0^\smallfrown {(\beta,\BB)} \in \sigma$ and there is exactly one $s_2 \vartriangleleft s_1$\footnote{We remark that $s_2 \vartriangleleft s_1$ iff $s_1 = s_2 \res (\vp{s_2}-1)$.} in $\sigma$, $s_2 = s_1^\smallfrown {\CC}$ with $\CC \leq_{\ALL} \BB$ and $H_{\gamma^+}^\2 = H_{\gamma^+} \prec H_{\gamma^+}^\CC$. Moreover, $\sigma_{s_2}$ is a winning strategy for Player II (\emph{Resurrect}) in $\GGG^{\wRA}$ after $\ap{(\beta,\CC),\CC}$, hence $V^\CC \models \wRA_\beta(\Gamma)$.

			Conversely, suppose that for all $\beta < \alpha$, $\BB \in \Gamma$ there is a $\CC$ such that $\wRA_\beta(\Gamma)$ holds in $V^\CC$ and $H_{\gamma^+} \prec H_{\gamma^+}^\CC$. Assume towards a contradiction that $\wRA_\alpha(\Gamma)$ fails. Then by $\ADF$ Player I (\emph{Kill}) has a winning strategy $\sigma$ in $\GGG^{\wRA}$ after $s_0 = \ap{(\alpha,\2),\2}$, and there is exactly one $s_1 = s_0^\smallfrown {(\beta,\BB)} \in \sigma$. Let $\CC$ be such that $\wRA_\beta(\Gamma)$ holds in $V^\CC$, $H_{\gamma^+} \prec H_{\gamma^+}^\CC$. Then $s_2 = s_1^\smallfrown {\CC}$ is a valid move (hence is in $\sigma$) and by the first part of this proof, since $V^\CC \models \wRA_\beta(\Gamma)$, Player II (\emph{Resurrect}) wins the game $\GGG^{\wRA} \res \ap{(\beta, \CC), \CC}$. Since $\sigma_{s_2}$ is a winning strategy for Player I (\emph{Kill}) in the same game, we get a contradiction.
		\end{description}
		Similar arguments yield the thesis also for $\RA_\alpha(\Gamma)$.
	\end{proof}

	\begin{remark} \label{rmk:firstorder}
		Assume $\ap{V,\mathcal{C}}$ is a model of $\MK$. Let $\phi_0(x, y,\gamma)$ be the formula
		\[
		\phi_0(x, y,\gamma) ~\equiv~ x,\, y\text{ are complete Boolean algebras and }H_{{\gamma^+}}^x \prec H_{{\gamma^+}}^y
		\]
		and for all $n < \omega$, let $\phi_{n+1}(x, y,\gamma)$ be the formula
		\[
		\phi_{n+1}(x, y,\gamma) ~\equiv~ \phi_0(x, y,\gamma) ~ \wedge \cp{\forall z \leq_\Gamma y ~ \exists w \leq_\Gamma z ~ \phi_n(y, w,\gamma)}.
		\]
		Then for all $n < \omega$ the assertion  $\phi_{n}(\2, \2,\gamma)$ is equivalent to $\RA_{n}(\Gamma)$ in $\ap{V,\mathcal{C}}$ and it is a formula with no class quantifier. In particular we get that
		\[
		\ap{V,\mathcal{C}} \models \RA_n(\Gamma) \quad \Longleftrightarrow \quad V \models \phi_{n}(\2,\2,\gamma).
		\]
		hence the formulae $\phi_{n}(\2,\2,\gamma)$ can be used as a first-order formulation of the axioms 
		$\RA_n(\Gamma)$ expressible in $\ZFC$ with no sort for class variables (provided $\Gamma,\gamma$ are both definable in 
		$\ZFC$). However, if 
		\[
		\ap{V,\mathcal{C}}\models\RA_\omega(\Gamma)
		\] 
		we can infer that for all $n < \omega$, $V \models \phi_{n}(\2,\2)$ but it is not at all clear whether 
		we can express in the structure $V$ that $\RA_\omega(\Gamma)$ holds. In fact, the 
		simplest strategy to express 
		this property of $V$ would require us to perform an infinite conjunction of the formulae 
		$\phi_{n}(\2,\2)$ for all $n<\omega$, thus getting out of first-order syntax.

		This problem can be sidestepped in models of $\MK$ appealing to the class-game formulation of these axioms.
	\end{remark}

	From now on, we focus on the recursive formulation\footnote{Recall that we assume to work in $\MK$, thus $\ADF$ holds and the recursive formulation is equivalent to the definition.} of the $\alpha$-resurrection axioms given by the latter proposition in order to prove the main results by induction on $\alpha$. The same will be done with the subsequent Definitions \ref{def:uplifting}, \ref{def:menas} of other class games.
	Note that by Proposition \ref{prop:radef}, $\wRA_0(\Gamma)$, $\RA_0(\Gamma)$ hold vacuously true for any $\Gamma$, hence $\wRA_1(\Gamma)$, $\RA_1(\Gamma)$ imply the non-iterated formulations of resurrection axioms given in \cite{hamkins:resurrection_uplifting} provided that $\gamma^+ \geq \cc$.

	These different forcing axioms are connected by the following implications:
	\begin{itemize}
		\item if $\beta < \alpha$, $\wRA_\alpha(\Gamma) \Rightarrow \wRA_\beta(\Gamma)$ 
		(same with $\RA$),
		\item if $\Gamma \subseteq \Delta$ and $\gamma_{\Gamma} \leq \gamma_\Delta$, 
		$\wRA(\Delta) \Rightarrow \wRA(\Gamma)$,
		\item $\RA_\alpha(\Gamma) \Rightarrow \wRA_\alpha(\Gamma)$, since we assume that $\Gamma$ is closed under two step iterations: the winning strategy $\sigma$ for II in $\GGG^{\RA}$ starting from $\ap{(\alpha, \2), \2}$ can also be used in $\GGG^{\wRA}$ and will force I to play always a $\BB_n$ in $\Gamma$. In particular $\sigma$ will remain a winning strategy also in $\GGG^{\wRA}$.
	\end{itemize}

	We are mainly interested in $\RA_\alpha(\Gamma)$, even though $\wRA_\alpha(\Gamma)$ will be convenient to state certain theorems in a modular form (thanks to its monotonic behavior with respect to $\Gamma$).
	Some relevant implications can be drawn between iterated resurrection axioms, Woodin's generic absoluteness for $L(\mathbb{R})$, and the usual forcing axioms.

	\begin{proposition} \label{prop:woodinrall}
		Assume there are class-many Woodin cardinals. Then $\RA_\ON(\ALL)$ holds.
	\end{proposition}
	\begin{proof}
		With these assumptions, we know that $H_{\omega_1}\prec H_{\omega_1}^\BB$ for any $\BB$. 
		Hence a winning strategy (among many) for II in $\GGG^{\RA(\ALL)}$ is to always play the same Boolean algebra I has played in the preceding move until I has to play the ordinal $0$. Then II can still move, but I cannot
		and loses.
	\end{proof}

	\begin{theorem}
		$\wRA_1(\Gamma)$ implies $H_{\gamma^+} \prec_1 V^\BB$ for all $\BB \in \Gamma$ and $\BFA_\kappa(\Gamma)$ for all $\kappa < \gamma^+$.
	\end{theorem}
	\begin{proof}
		Let $\BB$ be any Boolean algebra in $\Gamma$, and $\CC \leq_{\ALL} \BB$ be such that $H_{\gamma^+} \prec H_{\gamma^+}^\CC$, hence $H_{\gamma^+} \prec H_{\gamma^+}^\CC \prec_1 V^\CC$ by Levy's absoluteness. Let $\phi \equiv \exists x \psi(x)$ be a $\Sigma_1$ formula. If $\phi$ holds in $H_{\gamma^+}$, it trivially holds in $V^\BB$ since $\Sigma_1$ formulae are upwards absolute. If $\phi$ holds in $V^\BB$, it holds in $V^\CC$ as well hence in $H_{\gamma^+}$, concluding the first part.

		Let now $\BB$ be in $\Gamma$ and $\DDD$ be a family of $\kappa$-many predense subsets of 
		$\BB$ of size at most $\kappa$. Let $\BB' \subseteq \BB$ be the Boolean algebra (possibly not in $\Gamma$) 
		finitely generated by $\bigcup \DDD$ in $\BB$, 
		so that $\vp{\BB'} \leq \kappa$. Without loss of generality we can assume that both $\BB'$ and 
		$\DDD$ are in $H_{\kappa^+} \subseteq H_{\gamma^+}$ by replacing $\BB$ with 
		an isomorphic copy if necessary. Let $G$ be a $V$-generic filter for $\BB$. 
		Then $G$ meets every predense in $\DDD$, that is,
		\[
		V[G] \models \exists F \subseteq \BB' \text{ filter } \wedge \forall A \in \DDD ~ F \cap A \neq \emptyset
		\]
		and since $H_{\gamma^+} \prec_1 V^\BB$, $H_{\gamma^+}$ has to model the same completing the proof.
	\end{proof}
	
	\begin{theorem}
		Assume there are class-many super huge cardinals.\footnote{A cardinal $\kappa$ is \emph{super huge} iff for every ordinal $\alpha$ there exists an elementary embedding $j: V \to M\subseteq V$ with $\crit(j) = \kappa$, $j(\kappa) > \alpha$ and ${}^{j(\kappa)}M \subseteq M$.} 
		
		Then $\MM^{+++}$ implies $\RA_{\ON}(\SSP)$.
	\end{theorem}
	The proof of this theorem requires the reader to be familiar with the second 
	author's work~\cite{viale:category_forcing}.
	\begin{proof}
		Recall that $\gamma_{\SSP} = \omega_1$ and $\MM$ implies that $2^\omega = 2^{\omega_1} = \omega_2$. 
		We prove that $\MM^{+++}$ implies $\RA_\alpha(\SSP)$ by induction on $\alpha$. 
		For $\alpha = 0$ there is nothing to prove, 
		suppose now that $\alpha > 0$ and the thesis holds for all $\beta<\alpha$.

		Let $A$ be the class of all super huge cardinals in $V$. Let $\UU^{\SSP}_\delta$ be the forcing whose condition are the $\SSP$ complete Boolean algebras in $\SSP\cap V_\delta$ ordered by $\leq_{\SSP}$. Since $A$ is a proper class, by \cite[Thm. 3.5, Lemma 3.12]{viale:category_forcing} the class $\bp{\UU^{\SSP}_\delta: ~ \delta \in A}$ is predense in $(\SSP,\leq_{\SSP})$. Moreover $H_{2^{\omega_1}} = H_{\omega_2} \prec H_{\omega_2}^{\UU^{\SSP}_\delta}$ for all $\delta\in A$, since by \cite[Lemma 5.19]{viale:category_forcing} $\UU^{\SSP}_\delta$ is forcing equivalent to a (strongly) presaturated tower for any such $\delta$. Finally, by \cite[Cor. 5.20]{viale:category_forcing}, every such $\UU^{\SSP}_\delta$ forces $\MM^{+++}$ and preserves that there are class-many super huge cardinals (since these large cardinals are indestructible by small forcings). It follows by inductive hypothesis that every such $\UU^{\SSP}_\delta$ forces $\RA_\beta(\SSP)$ for any $\beta < \alpha$ as well; hence $\RA_\alpha(\SSP)$ holds in $V$.
	\end{proof}

	We remark that similar results were obtained by Hamkins and Johnstone from their formulation of the resurrection axiom $\RA(\Gamma)$ (see Definition \ref{def:hj_ra}). In fact, $\RA_1(\Gamma)$ is almost the same as their axiom $\RA(\Gamma)$ whenever $\gamma=\omega_1$ and $\gamma^+ = \omega_2 = \cc$.
	However, even in this case, there are some subtle differences. For instance, let $\Gamma$ be such that $\gamma = \omega_1$, $2^\omega = 2^{\omega_1} = \omega_2$ and $\Add(\omega_1, 1)$ is in $\Gamma$ (e.g., $\Gamma$ is among countably closed, axiom-$A$, proper, semiproper, $\SSP$). Then:
	\begin{itemize}
	\item
		$\Add(\omega_1,1)$ preserves $\RA(\Gamma)$. Assume that $\RA(\Gamma)$ holds in $V$ and let $G$ be $V$-generic for $\Add(\omega_1,1)$  so that $2^\omega=\omega_1$ in $V[G]$ and $H_{\cc}^{V[G]}=H_{\omega_1}^{V[G]}=H_{\omega_1}^{V}$. Then, any $\BB$ in $\Gamma^{V[G]}$ which resurrects the theory of $H_{\cc}^V$ will also resurrect the theory of $H_{\cc}^{V[G]}=H_{\omega_1}^{V[G]}=H_{\omega_1}^{V}$. Since there are densely many 
		$\BB$ resurrecting the theory of $H_\cc^V$ in $V$, we get that there are densely many $\BB$ in $\Gamma^{V[G]}$
		resurrecting the theory of $H_{\omega_1}^{V[G]}$ .
		\item
		There is no reason to expect that $\Add(\omega_1,1)$ preserves $\RA_1(\Gamma)$. 
		In fact, in this latter case we want to resurrect the theory of $H_{\omega_1^+}$ as computed in $V[G]$ and 
		$H_{\omega_1^+}^{V[G]} \supseteq H_{\omega_2^+}^V$ since $\Add(\omega_1,1)$ always collapses 
		$\omega_1^{\omega}$, which is $\omega_2$ in our case. But the theory of 
		$H_{\omega_2^+}^V$ is by no means controlled by $\RA_1(\Gamma)$.
	\end{itemize}


\subsection{Resurrection axioms and generic absoluteness}

	The main motivation for the iterated resurrection axioms can be found in the following result.

	\begin{theorem} \label{thm:absoluteness}
		Suppose $n \in \omega$, $\Gamma$ is well behaved, 
		$\RA_n(\Gamma)$ holds, and $\BB \in \Gamma$ forces $\RA_n(\Gamma)$. 
		Then $H_{\gamma^+} \prec_n H_{\gamma^+}^\BB$ (where $\gamma=\gamma_\Gamma$).
	\end{theorem}
	\begin{proof}
		We proceed by induction on $n$. Since $\gamma^+ \leq (\gamma^+)^{V^\BB}$, $H_{\gamma^+} \subseteq H_{\gamma^+}^\BB$ and the thesis holds for $n = 0$ by the fact that for all transitive structures $M$, $N$, if $M \subset N$ then $M \prec_0 N$.
		 Suppose now that $n > 0$, and fix $G$ $V$-generic for $\BB$. By Proposition \ref{prop:upliftingdef} and $\RA_n(\Gamma)$, let $\CC \in V[G]$ be such that whenever $H$ is $V[G]$-generic for $\CC$, $V[G \ast H] \models \RA_{n-1}(\Gamma)$ and $H_{\gamma^+}^V \prec H_{\gamma^+}^{V[G \ast H]}$. 
		 Hence we have the following diagram:
		\[
			\begin{tikzpicture}[xscale=1.5,yscale=-1.2]
				\node (A0_0) at (0, 0) {$H_{\gamma^+}^V$};
				\node (A0_2) at (2, 0) {$H_{\gamma^+}^{V[G \ast H]}$};
				\node (A1_1) at (1, 1) {$H_{\gamma^+}^{V[G]}$};
				\path (A0_0) edge [->]node [auto] {$\scriptstyle{\Sigma_\omega}$} (A0_2);
				\path (A1_1) edge [->]node [auto,swap] {$\scriptstyle{\Sigma_{n-1}}$} (A0_2);
				\path (A0_0) edge [->]node [auto,swap] {$\scriptstyle{\Sigma_{n-1}}$} (A1_1);
			\end{tikzpicture}
		\]
		obtained by inductive hypothesis applied both on $V$, $V[G]$ and on $V[G]$, $V[G \ast H]$ since in all those classes $\RA_{n-1}(\Gamma)$ holds.

		Let $\phi \equiv \exists x \psi(x)$ be any $\Sigma_{n}$ formula with parameters in $H_{{\gamma^+}}^V$. First suppose that $\phi$ holds in $V$, and fix $\bar{x} \in V$ such that $\psi(\bar{x})$ holds. Since $H_{\gamma^+}^V \prec_{n-1} H_{\gamma^+}^{V[G]}$ and $\psi$ is $\Pi_{n-1}$, it follows that $\psi(\bar{x})$ holds in $V[G]$ hence so does $\phi$.
		Now suppose that $\phi$ holds in $V[G]$ as witnessed by $\bar{x} \in V[G]$. Since $H_{\gamma^+}^{V[G]} \prec_{n-1} H_{\gamma^+}^{V[G \ast H]}$ it follows that $\psi(\bar{x})$ holds in $V[G \ast H]$, hence so does $\phi$. Since $H_{\gamma^+}^V \prec H_{\gamma^+}^{V[G \ast H]}$, the formula $\phi$ holds also in $V$ concluding the proof.
	\end{proof}

	\begin{corollary} \label{cor:absoluteness}
		If $\Gamma$ is well behaved, $\RA_\omega(\Gamma)$ holds, and $\BB \in \Gamma$ forces $\RA_\omega(\Gamma)$, then  $H_{\gamma^+} \prec H_{\gamma^+}^\BB$ (where $\gamma=\gamma_\Gamma$).
	\end{corollary}

	A generic absoluteness result for $L(\ON^{\omega_1})$ and $\SSP$ forcings preserving $\MM^{+++}$ 
	has been obtained by the second author in \cite{viale:category_forcing}. 
	Corollary~\ref{cor:absoluteness} for $\Gamma = \SSP$ provides a weaker statement than the ones obtained in 
	\cite{viale:category_forcing} for the $\SSP$ forcings 
	(since it concerns the smaller model $H_{\omega_2}$ rather than the whole $L(\ON^{\omega_1})$). 
	However ~\ref{cor:absoluteness} is a general result since it holds also for interesting choices of 
	$\Gamma \neq \SSP$. Moreover the consistency of $\RA_\omega(\Gamma)$ 
	follows from much weaker large cardinal assumptions than the one needed for $\MM^{+++}$. 
	Furthermore, the above result can be applied to classes $\Gamma$ with $\gamma_\Gamma>\omega_1$, 
	providing insights on how to get generic absoluteness for the theory of $H_{\gamma^+}$ 
	in case $\gamma^+ \geq \omega_3$.

	\section{Uplifting cardinals and definable Menas functions} \label{sec:uplifting}

	We introduce the $(\alpha)$-uplifting cardinals in order to obtain the consistency of the 
	$\alpha$-iterated resurrection axioms.
	\begin{definition} \label{def:uplifting}
		The \emph{uplifting game} $\GGG^{\UP}$ is as follows. Player I (\emph{challenge}) plays couples of ordinals $(\alpha_n, \theta_n)$ such that $\alpha_{n+1} < \alpha_n$. Player II (\emph{uplift}) plays inaccessible cardinals $\kappa_n$ such that $V_{\kappa_n+1} \prec_{\Delta^1_1} V_{\kappa_{n+1}+1}$ and\footnote{We recall that the notation $M \prec_{\Delta^1_1} N$ is explained in Section \ref{sec:class_games}.} $\kappa_n \geq \theta_n$. The last player who can move wins.
			\[
			\xymatrixcolsep{1pc}
			
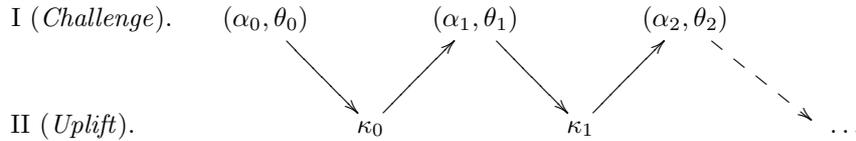
\xymatrix{
			\text{I (\emph{Challenge}).} &  (\alpha_0, \theta_0) \ar@{->}[rd] & & (\alpha_1, \theta_1) \ar@{->}[rd] & & (\alpha_2, \theta_2) \ar@{-->}[rd] & \\
			\text{II (\emph{Uplift}).\phantom{eeii}} & & \kappa_0 \ar@{->}[ru] & & \kappa_1 \ar@{->}[ru] & & \quad \quad \ldots
			}
			\]

		We say that $\kappa$ is $(\alpha)$-uplifting iff Player II (\emph{Uplift}) wins the uplifting game after $\ap{(\alpha,0), \kappa}$. We say that $\kappa$ is $(\ON)$-uplifting iff it is $(\alpha)$-uplifting for all $\alpha \in \ON$.
	\end{definition}

	The uplifting game is a clopen class game, since Player I (\emph{Challenge}) 
	plays a descending sequence of ordinals. It follows that upliftingness is a $\Delta^1_1$-property under $\ADF$ (see Proposition \ref{prop:winningdelta}). 
	As for the iterated resurrection axioms, we can give a formulation of $(\alpha)$-upliftingness in recursive terms.

	\begin{proposition}[$\ADF$] \label{prop:upliftingdef}
		$\kappa$ is $(\alpha)$-uplifting iff it is inaccessible and for all $\beta < \alpha$ and $\theta > \kappa$ there is a $\lambda > \theta$ that is $(\beta)$-uplifting and $V_{\kappa+1} \prec_{\Delta^1_1} V_{\lambda+1}$.
	\end{proposition}
	\begin{proof}
		Similarly as in Proposition \ref{prop:radef}. Let $\sigma$ be a winning strategy for Player I or II in $\GGG^{\UP}$, $s = \ap{(\alpha,0), \kappa, (\beta,\theta), \lambda}$ be in $\sigma$. Then $\sigma_s = \bp{\ap{(\beta,0),\lambda}^\smallfrown u: ~ s^\smallfrown u \in \sigma}$ is a winning strategy for the same player in $\GGG^{\UP} \res \ap{(\beta,0),\lambda}$. Then we can follow step by step the last part of the proof of Proposition \ref{prop:radef} using the reduction $\sigma \to \sigma_s$ defined above.
	\end{proof}

	We remark that the definition of $(0)$-uplifting cardinal coincides with that of inaccessible cardinal,
	while replacing ``inaccessible'' with ``regular'' in Definition \ref{def:uplifting} makes 
	no difference for $\alpha > 0$, since a successor cardinal $\kappa$ cannot satisfy 
	$V_{\kappa+1} \prec_{\Delta^1_1} V_{\lambda+1}$ for any $\lambda > \kappa$. 

	Although similar, the notion of $(1)$-uplifting cardinal is stronger than the notion
	of \emph{uplifting cardinal} in Johnstone and Hamkins \cite{hamkins:resurrection_uplifting}.
	However, consistency-wise $(\ON)$-uplifting cardinals are very close to
	Johnstone and Hamkins uplifting cardinals, as shown in 
	Proposition \ref{pr:upliftingmahlo} below.	
	
	The key reason which led us to introduce $(\alpha)$-uplifting cardinals as a natural (second-order) 
	strengthening of 
	the Hamkins and Johnstone original (first-order) notion of upliftingness is to be found in 
	Lemma~\ref{lm:upliftingrelativization} below which states a nice reflection property of 
	$(\alpha)$-uplifting cardinals which we cannot predicate for the natural recursive strengthenings of 
	Hamkins and Johnstone notion of upliftingness.
	These reflection properties are a key ingredient in our proof of the consistency of the iterated resurrection 
	axioms. We will come back to these issues in more details in Section~\ref{sec:consistency}.


\subsection{Consistency strength of $(\alpha)$-uplifting cardinals}
	
	\begin{lemma} \label{lem:statdelta21}
		Assume $\delta$ is Mahlo. Then there are stationarily many inaccessible 
		$\kappa < \delta$ such that $V_{\kappa+1} \prec_{\Delta^1_1} V_{\delta+1}$.
	\end{lemma}
	\begin{proof}
		Let $C \subseteq \delta$ be a club, and let $M_0$ be the first-order Skolem hull of $\bp{C}$ in $V_{\delta+1}$.
		Define a sequence $\ap{(M_\alpha, \kappa_\alpha): ~ \alpha < \delta}$ where $\kappa_\alpha = \max(\alpha, \rank(M_\alpha \cap V_\delta))$; $M_\alpha = \bigcup_{\beta < \alpha} M_\beta$ for  limit ordinals $\alpha$; and $M_{\alpha+1} \prec V_{\delta+1}$ is obtained applying the Lowenheim-Skolem Theorem to $M_\alpha \cup V_{\kappa_\alpha}$, so that $M_{\alpha+1} \supseteq M_\alpha \cup V_{\kappa_\alpha}$ and $\vp{M_{\alpha+1}} = \vp{M_\alpha \cup V_{\kappa_\alpha}}$.

		Since $\vp{M_\alpha} < \delta$ implies that $\kappa_\alpha < \delta$ and $\vp{M_{\alpha+1}} = \vp{M_\alpha \cup V_{\kappa_\alpha}} < \delta$, by induction on $\alpha$ we have that $\vp{M_\alpha}, \kappa_\alpha < \delta$ for all $\alpha$.
		Furthermore, the sequence $\ap{\kappa_\alpha: ~ \alpha <\delta}$ is a club on $\delta$, which is Mahlo, thus we can find an $\bar{\alpha} < \delta$ limit such that $\bar{\alpha}=\kappa_{\bar{\alpha}}$ is inaccessible.

		Since $\bar{\alpha}$ is limit, $M_{\bar{\alpha}} \cap V_\delta = V_{\kappa_{\bar\alpha}}$. Since $V_{\delta+1} \models ``C \text{ is a club''}$ and $C \in M_{\bar{\alpha}}$, $M_{\bar{\alpha}} \prec V_{\delta+1}$, we get that $M_{\bar{\alpha}} \models ``C \text{ is a club''}$ hence $\kappa_{\bar{\alpha}}$ is a limit point of $C$. Thus, $\kappa_{\bar{\alpha}}$ is an inaccessible cardinal in $C$ and by Proposition \ref{prop:delta21abs} $V_{\kappa_{\bar{\alpha}}+1} \equiv_{\Delta^1_1(T)} M_{\bar{\alpha}} \prec V_{\delta+1}$, concluding the proof.
	\end{proof}

	\begin{proposition} \label{pr:upliftingmahlo}
		Assume $\delta$ is Mahlo. Then $V_{\delta+1}$ models $\MK+$ there are class-many $(\ON)$-uplifting cardinals.
	\end{proposition}
	\begin{proof}
		Since $\delta$ is inaccessible, $V_{\delta+1}$ models $\MK$ and hence $\AD(\Delta^0_1)$. 
		Furthermore,
		\[
		S = \bp{\kappa < \delta: ~  \kappa \text{ inaccessible} ~\wedge~ V_{\kappa+1} \prec_{\Delta^1_1} V_{\delta+1}}
		\]
		is stationary by Lemma \ref{lem:statdelta21}. We prove that every element of $S$ is $(\alpha)$-uplifting in $V_{\delta+1}$ by induction on $\alpha < \delta$.

		First, every element of $S$ is $(0)$-uplifting by definition. Suppose now that every element of $S$ is $(\beta)$-uplifting for every $\beta < \alpha$, and let $\kappa$ be in $S$. Since $S$ is unbounded, for every $\beta < \alpha$, $\theta > \kappa$ in $V_\delta$ there is a $\lambda \in S$, $\lambda > \theta$.
		Such cardinal $\lambda$ is $(\beta)$-uplifting by inductive hypothesis and $V_{\kappa+1}, V_{\lambda+1} \prec_{\Delta^1_1} V_{\delta+1}$ implies that $V_{\kappa+1} \prec_{\Delta^1_1} V_{\lambda+1}$. Since $V_{\delta+1} \models \ADF$ we can use Proposition \ref{prop:upliftingdef} to conclude.
	\end{proof}
	
	As shown in \cite[Thm. 11]{hamkins:resurrection_uplifting}, if there is an uplifting cardinal, 
	there is a transitive model of $\ZFC + ``\ON \text{ is Mahlo''}$. So the existence of an $(\ON)$-uplifting cardinal is in consistency strength strictly between the existence of a Mahlo cardinal and the scheme ``$\ON \text{ is Mahlo}$''. We take these bounds to be rather close together and low in the large cardinal hierarchy.


\subsection{Reflection properties of $(\alpha)$-uplifting cardinals}		

The following proposition outlines a key reflection property of $(\alpha)$-uplifting cardinals.
	
	\begin{lemma}[$\ADF$] \label{lm:upliftingrelativization}
		Let $\kappa$ be an $(\alpha)$-uplifting cardinal with $\alpha < \kappa$, and let $\delta < \kappa$. Then $\cp{\delta \text{ is }(\alpha)\text{-uplifting}}^{V_{\kappa+1}}$ iff it is $(\alpha)$-uplifting.
	\end{lemma}
	\begin{proof}
		Let $\phi(\alpha)$ be the statement of this theorem, i.e:
		\[
		\forall \kappa > \alpha ~ (\alpha)\text{-uplifting} ~ \forall \delta < \kappa ~ \cp{\cp{\delta \text{ is }(\alpha)\text{-uplifting}}^{V_{\kappa+1}} \Leftrightarrow \delta \text{ is }(\alpha)\text{-uplifting}}
		\]
		We prove $\phi(\alpha)$ by induction on $\alpha$ using the recursive formulation given by Proposition \ref{prop:upliftingdef}. For $\alpha = 0$ it is easily verified, suppose now that $\alpha > 0$.

		For the forward direction, suppose that $\cp{\delta \text{ is }(\alpha)\text{-uplifting}}^{V_{\kappa+1}}$, 
		and let $\beta < \alpha$, $\theta > \delta$ be ordinals. Let $\lambda > \theta$ be a $(\beta)$-uplifting 
		cardinal with $V_{\kappa+1} \prec_{\Delta^1_1} V_{\lambda+1}$, so that 
		$\cp{\delta \text{ is }(\alpha) \text{-uplifting}}^{V_{\lambda+1}}$ since $(\alpha)$-upliftingness is a 
		$\Delta^1_1$-property under $\ADF$ (and $\ADF$ holds at inaccessible cardinals).
		Then there is a $\nu > \theta$ in $V_{\lambda+1}$ with $V_{\delta+1} \prec_{\Delta^1_1} V_{\nu+1}$ 
		and $\cp{\nu \text{ is }(\beta)\text{-uplifting}}^{V_{\lambda+1}}$. 
		By inductive hypothesis, since $\beta<\alpha$ and $\lambda$ is $(\beta)$-uplifting in $V$, 
		also $\nu$ is $(\beta)$-uplifting in $V$.

		Conversely, suppose that $\delta$ is $(\alpha)$-uplifting in $V$ and let $\beta < \alpha$, $\theta > \delta$ be ordinals in $V_\kappa$. Let $\nu > \theta$ be a $(\beta)$-uplifting cardinal such that $V_{\delta+1} \prec_{\Delta^1_1} V_{\nu+1}$, and let $\lambda > \nu$ be a $(\beta)$-uplifting cardinal such that $V_{\kappa+1} \prec_{\Delta^1_1} V_{\lambda+1}$. By inductive hypothesis, since $\beta < \alpha$ and $\nu$, $\lambda$ are $(\beta)$-uplifting in $V$, $\cp{\nu \text{ is }(\beta)\text{-uplifting}}^{V_{\lambda+1}}$ thus
		\[
		V_{\lambda+1} \models \exists \nu > \theta ~ V_{\delta+1} \prec_{\Delta^1_1} V_{\nu+1} \wedge \nu \text{ is }(\beta)\text{-uplifting}.
		\]
		Remark that upliftingness can be expressed by a canonical $\Delta^1_1$-formula (see Proposition \ref{prop:winningdelta}). Remark also that the above formula is expressed through conjunctions and set quantifications over the formula defining upliftingness and other first-order statements (such as $V_{\delta+1} \prec_{\Delta^1_1} V_{\nu+1}$). Hence the above formula is as well canonical $\Delta^1_1$ (see Proposition \ref{prop:canonicalformula}). By $\Delta^1_1$-elementarity, $V_{\kappa+1}$ models the same formula thus concluding the proof.
	\end{proof}


\subsection{Menas functions for uplifting cardinals} \label{sec:menas}

	To obtain the consistency of the $(\alpha)$-iterated resurrection axioms
	 we will use a lottery iteration relative to a fast-growing function $f: \kappa \to \kappa$ for a sufficiently large cardinal $\kappa$. The exact notion of fast-growth we will need is given by the \emph{Menas} property schema introduced in \cite{menas:strong_compactness} and developed by Hamkins for several different cardinal notions in \cite{hamkins:lottery_preparation,hamkins:resurrection_uplifting}.

	We remark that it is always possible to define Menas functions for cardinals that have a Laver function, while it is also possible to define such functions for some cardinals that don't have a Laver function. Moreover, from Menas functions we can obtain many of the interesting consequences given by Laver functions.

	\begin{definition} \label{def:menas}
		The \emph{Menas uplifting game} $\GGG^{M-\UP}$ is as follows. Player I (\emph{Challenge}) plays couples of ordinals $(\alpha_n, \theta_n)$ such that $\alpha_{n+1} < \alpha_n$. Player II (\emph{Uplift}) plays partial functions $f_n: \kappa_n \to \kappa_n$ with $\kappa_n$ inaccessible such that%
		\footnote{$\ap{M, C} \prec_{\Delta^1_1} \ap{N,D}$ iff $\set(M)\subseteq \set(N)$ and for all $\Delta^1_1$-properties $\phi(\vec{x},Y)$ with $\vec{x}$ a tuple of set variables and $Y$ a class variable, $\ap{M,C}\models\phi(\vec{a},C)$ iff $\ap{N,D}\models\phi(\vec{a},D)$.}
		$\ap{V_{\kappa_n+1}, f_n} \prec_{\Delta^1_1} \ap{V_{\kappa_{n+1}+1}, f_{n+1}}$ and $f_{n+1}(\kappa_n) \geq \theta_{n+1}$. The last player who can move wins.
			\[
			\xymatrixcolsep{1pc}
			
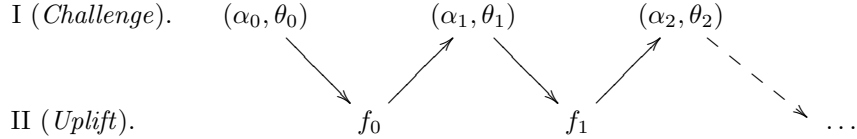
\xymatrix{
			\text{I (\emph{Challenge}).} &  (\alpha_0, \theta_0) \ar@{->}[rd] & & (\alpha_1, \theta_1) \ar@{->}[rd] & & (\alpha_2, \theta_2) \ar@{-->}[rd] & \\
			\text{II (\emph{Uplift}).\phantom{eeii}} & & f_0 \ar@{->}[ru] & & f_1 \ar@{->}[ru] & & \quad \quad \ldots
			}
			\]

		We say that a partial function $f: \kappa \to \kappa$ is Menas for $(\alpha)$-uplifting iff Player II (\emph{Uplift}) wins the Menas uplifting game after $\ap{(\alpha,0), f}$. We say that $f$ is Menas for $(\ON)$-uplifting iff it is Menas for $(\alpha)$-uplifting for all $\alpha \in \ON$.
	\end{definition}

	\begin{proposition}[$\ADF$]
		A partial function $f : \kappa \to \kappa$ with $\kappa$ inaccessible is Menas for $(\alpha)$-uplifting iff for all $\beta < \alpha$ and $\theta > \kappa$ there is a Menas for $(\beta)$-uplifting function $g: \lambda \to \lambda$ with $g(\kappa) > \theta$ and $\ap{V_{\kappa+1},f} \prec_{\Delta^1_1} \ap{V_{\lambda+1},g}$.
	\end{proposition}
	\begin{proof}
		Similarly as in Proposition \ref{prop:upliftingdef}.
	\end{proof}

	We can now prove the existence of definable Menas functions for $(\alpha)$-uplifting cardinals.
	
	\begin{proposition}[$\ADF$] \label{lm:menas}
		If $\kappa$ is $(\alpha)$-uplifting, then there is a 
		canonically $\Delta^1_1$-definable\footnote{More precisely, there is a canonically $\Delta_1^1$-formula $\phi(X)$ such that $V_{\kappa+1}\models\phi(f) \wedge \exists! X \phi(X)$.} Menas function for $(\alpha)$-uplifting on $\kappa$.
	\end{proposition}
	\begin{proof}
		We prove by induction on $\alpha$ that whenever $\kappa$ is $(\alpha)$-uplifting, such a function is given by relativizing to $V_{\kappa+1}$ the following definable class function (failure of upliftingness function):
		\[
		\begin{split}
		f(\xi) = \sup&\left\{\nu: ~ V_{\xi+1} \prec_{\Delta^1_1} V_{\nu+1} ~~\wedge \right. \\
		&\quad \left. \exists \beta ~ \nu \text{ is }(\beta)\text{-uplifting} ~\wedge~ \xi \text{ is not }(\beta+1)\text{-uplifting}\right\}.
		\end{split}
		\]
		Note that $f(\xi)$ is undefined only if $\xi$ is $(\ON)$-uplifting (in the domain of $f$), since otherwise $\xi$ would be an element considered in the supremum. Thus $f(\xi) \geq \xi$ when it is defined.
		If $\kappa$ is $(0)$-uplifting any function $g: \kappa \to \kappa$ is Menas, in particular our $f$. Now suppose that $\kappa$ is $(\alpha)$-uplifting with $\alpha > 0$ and let $\beta < \alpha$, $\theta > \kappa$ be ordinals. Let $\nu > \theta$ be a $(\beta)$-uplifting cardinal such that $V_{\kappa+1} \prec_{\Delta^1_1} V_{\nu+1}$, and let $\lambda$ be the least $(\beta)$-uplifting cardinal bigger than $\nu$ such that $V_{\kappa+1} \prec_{\Delta^1_1} V_{\lambda+1}$.

		Thus no $\nu' \in (\nu,\lambda)$ with $V_{\kappa+1} \prec_{\Delta^1_1} V_{\nu'+1}$ can be $(\beta)$-uplifting in $V$, hence by Lemma~\ref{lm:upliftingrelativization} neither in $V_\lambda$. It follows that $\kappa$ cannot be $(\beta+1)$-uplifting in $V_{\lambda}$, while again by Lemma \ref{lm:upliftingrelativization} it is $(\beta)$-uplifting in $V_{\lambda}$. Then any $\nu'$ considered in calculating $f^{V_{\lambda+1}}(\kappa)$ must be witnessed by a 
		$\beta' \geq \beta$ (since $\kappa$ is $(\beta'+1)$-uplifting in $V_{\lambda+1}$ for any $\beta'<\beta$) hence must be $(\beta)$-uplifting. It follows that $f^{V_{\lambda+1}}(\kappa) = \nu$.
		
		Since the graph of $f$ is defined through conjunctions and set quantifications over upliftingness (which is a canonical $\Delta^1_1$-property), it is as well canonical $\Delta^1_1$ (see Proposition \ref{prop:canonicalformula}) hence $\ap{V_{\kappa+1},f^{V_{\kappa+1}}} \prec_{\Delta^1_1} \ap{V_{\lambda+1},f^{V_{\lambda+1}}}$ concluding the proof.
	\end{proof}

	Note that in the proof of this Lemma we used in key steps the reflection properties of $(\alpha)$-uplifting cardinals given by Lemma~\ref{lm:upliftingrelativization}.

	\section{Consistency strength} \label{sec:consistency}

	The results of this section expand on the ones already present in~\cite{hamkins:resurrection_uplifting} and~\cite{viale:category_forcing}. In the previous section we outlined the large cardinal properties needed for our consistency proofs. Now we apply the machinery developed by Hamkins and Johnstone in their proof of the consistency of $\RA(\Gamma)$ for various classes of $\Gamma$, and show that with minor adjustments their techniques will yield the desired consistency results for $\RA_\alpha(\Gamma)$ for a weakly iterable well behaved $\Gamma$, when applied to lottery preparation forcings guided by suitable Menas functions. For this reason we feel free to sketch some of the proofs leaving to the reader to check the details, which follow closely what is done in~\cite{hamkins:resurrection_uplifting}.
	
	For the classes $\Gamma$ such that $\gamma^+ = \cc$ (which is the case for all the classes $\Gamma$ mentioned in this paper with $\gamma_\Gamma=\omega_1$) 
	$\RA_\alpha(\Gamma) \Rightarrow \RA(\Gamma)$. A lower bound for the latter is given by the axiom scheme ``$\ON$ is Mahlo'' \cite{hamkins:resurrection_uplifting}, hence the same lower bound can be predicated for $\RA_\alpha(\Gamma)$ as well for all these classes. This lower bound is rather close to the upper 
	bound we will obtain below for all relevant $\Gamma \neq \SSP$ (the existence of a Mahlo cardinal). 
	With some technical twists the same lower bounds can be inferred also for all the other classes $\Gamma$ for which $\gamma_\Gamma\neq\omega_1$, but we decided to skip the details and concentrate instead on the consistency proofs which are more delicate to handle. 
The reader interested in these details is referred to \cite{a:thesis}.


\subsection{Upper bounds}

	Let $\Gamma$ be a weakly iterable and well behaved class of forcing notions. 
	In this section we prove that $\PP^{\Gamma,f}_\kappa$, the lottery iteration of $\Gamma$ relative to a function 
	$f: \kappa \to \kappa$, forces $\RA_\alpha(\Gamma)$ whenever $f$ is Menas for $(\alpha)$-uplifting. 
	In order to prove this result, we will need to ensure that $\PP^{\Gamma,f}_\kappa$ ``behaves well'' 
	as a class forcing with respect to $V_{\kappa+1}$.

	There are two possible approaches. In the first one, we can consider $\ap{V_\kappa,f}$ as a $\ZFC$ model extended with an additional unary predicate for $f$, so that $\PP^{\Gamma,f}_\kappa$ can be handled as a definable class forcing in $V_\kappa$. Thus we can proceed following step by step the analogous argument carried out in \cite{hamkins:resurrection_uplifting}. The second approach considers the $\MK$ model $V_{\kappa+1}$ and expands on the results in \cite{antos:class_forcing} to prove that $\PP^{\Gamma,f}_\kappa$ preserves enough of $\MK$ and behaves well with respect to elementarity (though not all of $\MK$: e.g., the power set axiom is not preserved by $\PP^{\Gamma,f}_\kappa$). Even though the second approach is more general and natural in some sense, the first approach is considerably simpler (modulo certain complications regarding the definability of the forcing relation we will side-step). 
	Thus we will follow the first approach and give the following definition.

	\begin{definition}
		Let $\ap{M,C}$ be a model of $\ZFC$ expanded with an additional class predicate $C$, and let $\PP$ be a class partial order definable in $\ap{M,C}$.
		
		An $\ap{M,C}$-generic filter $G$ for $\PP$ is a filter meeting all dense subclasses of $\PP$ which are definable in $\ap{M,C}$ with parameters.
		
		$\PP$ is \emph{nice for forcing} in $\ap{M,C}$ if the forcing relation $\Vdash_\PP$ for formulae of first-order logic (i.e., with no class variables) is definable in $\ap{M,C}$ and the forcing theorem holds, i.e., for every first-order formula $\phi$ with parameters in $M \cup \bp{C}$ and every $\ap{M,C}$-generic filter $G$ for $\PP$,
		\[
		\ap{M[G],C} \models \phi ~\Leftrightarrow~ \exists p \in G ~ \ap{M,C} 
		\models \cp{p \Vdash_\PP \phi}.
		\]
	\end{definition}

	Note that the definition is interesting only when $\PP$ is a proper class forcing with respect to $M$, 
	and so $\PP \notin M$. The usual arguments regarding the existence or non-existence of 
	$\ap{M,C}$-generic filters apply: i.e. $\ap{V,C}$-generic filters do not exists when $V$ is the universe of sets and 
	$C$ is the collection of all proper classes, and $\ap{M,C}$-generic filters do exist if $M\cup C$ is a countable set.
	
	The next lemmas provide a sufficient condition for being nice for forcing in $H_\lambda$.
	The Lemmas hide certain delicate points which concern the definability of the forcing relation for class forcings in
	models of $\ZF$. While the forcing relation for a class forcing is always definable in models of 
	$\mathsf{MK}$ (see~\cite{antos:class_forcing}), this is not in general true for definable class forcings in
	models of $\ZFC$
	(see~~\cite[Theorem 1.3]{HOLKRALUCNJESCH16}),  but it is the case
	for the type of class forcings we are interested in (see~\cite{HOLKRASCH17}).

	\begin{lemma} \label{lm:hkpresat}
		Let $\lambda$ be a regular cardinal in $V$ and let $\PP \subseteq H_\lambda$ be a partial order 
		preserving the regularity of $\lambda$. Assume $G$ is $V$-generic for $\PP$. Then
		\[
		H_\lambda[G] =\{\val_G(\dot{x}): \dot{x} \in V^{\PP}\cap H_\lambda\}= H_\lambda^{V[G]}.
		\]
	\end{lemma}

	The above Lemma is rather standard but we sketch a proof since we cannot find a precise reference for it.

	\begin{proof}
		Since every element of $H_\lambda$ with $\lambda$ regular is coded with a bounded subset of $\lambda$, and $\PP$ preserves the regularity of $\lambda$, we can assume that every $\PP$-name for an element of $H_{\lambda}^{V[G]}$ is coded by a $\PP$-name for a function $\dot{f}:\lambda\to 2$ such that $\dot{f}$ is allowed to assume the value $1$ only on a bounded subset of $\lambda$. 
		In particular we let for any such $\dot{f}$,
		\[
		D_{\dot{f}}=\bp{p\in\PP: ~ \exists\alpha_p ~~ p \Vdash \dot{f}^{-1}[\{1\}]\subseteq\alpha_p}
		\]
		and for all $\xi<\lambda$,
		\[
		E_{\xi,\dot{f}} = \bp{p\in\PP: ~ \exists i<2 ~~ p \Vdash \dot{f}(\xi)=i}
		\]
		Notice that the above sets are open dense for any $\xi,\dot{f}$, and also that
		$p\in D_{\dot{f}}$ as witnessed by $\alpha_p$ implies that $p\in E_{\xi,\dot{f}}$ for all $\xi\geq\alpha_p$. 
		In particular to decide the values of $\dot{f}$ below any $p\in D_{\dot{f}}$ 
		we just need to consider the dense sets $E_{\xi,\dot{f}}$ for $\xi<\alpha_p$.

		Let $p \in \PP$ be arbitrary, and let $A_\xi\subseteq E_{\xi,\dot{f}}\cap D_{\dot{f}}$ be maximal antichains
		for all $\xi<\alpha_p$. Since $\PP$ preserves the regularity of $\lambda$, it is ${<}\lambda$ presaturated hence we can find $q\leq p$ such that $q\in D_{\dot{f}}$ as witnessed by $\alpha_p$ and
		\[
		B_\xi = \bp{ r\in A_\xi: r\text{ is compatible with }q}
		\]
		has size less than $\lambda$ for all $\xi<\alpha_p$. We can now use these antichains 
		$B_\xi$ to cook up a name $\dot{g}_q\in H_\lambda\cap V^{\PP}$ such that $q$ forces that $\dot{f}=\dot{g}_q$.
		By standard density arguments, the thesis follows.
	\end{proof}

	\begin{lemma} \label{lm:nice}
		Assume $\kappa$ is inaccessible.
		Let $\PP \subseteq H_\kappa$ be a partial order preserving the regularity of $\kappa$ that is definable in $\ap{H_\kappa,C}$. Then $\PP$ is nice for forcing in $\ap{H_\kappa,C}$.
	\end{lemma}
	\begin{proof}
		Since $\PP$ is a definable class in $\ap{H_\kappa,C}$, $\kappa$ is inaccessible and $\PP$ preserves 
	the regularity of $\kappa$, by~\cite[Theorem 2.4, Lemma 2.6]{HOLKRASCH17}, the corresponding forcing 
	relation $\Vdash_\kappa$ between elements of $\PP$ and formulae with parameters in $H_\kappa\cap V^{\PP}$ 
	whose quantifiers range only over the $\PP$-names in $H_\kappa$ is definable in $\ap{H_\kappa,C}$.
		Moreover, we can prove by induction on $\phi$ that this relation coincides with the forcing relation as 
		calculated in $V$, i.e., 
		\[
		V\models p \Vdash_\PP \phi^{\dot{H}_\kappa} ~ \Leftrightarrow ~ (H_\kappa,C)\models p \Vdash_\kappa \phi.
		\]
		
		First, assume that $\phi \equiv \sigma_0 \mathrel{R} \sigma_1$ is an atomic formula with 
		$\sigma_0, \sigma_1 \in H_\kappa$ and $R$ among $\in,=,\subseteq$. 
		We prove this case by induction on the pairs $\ap{\rank(\sigma_0),\rank(\sigma_1)}$ with the square order, 
		as in the proof of the forcing theorem. We handle with some care the case $\sigma_0 \in \sigma_1$ 
		and leave the other cases to the reader.
		\begin{align*}
			p \Vdash_\PP \sigma_0\in\sigma_1 &\Leftrightarrow \forall q \leq p ~ \exists \ap{\tau,s} \in \sigma_1 ~ 
			\exists r \leq q,s ~ r \Vdash_\PP \sigma_0 = \tau \Leftrightarrow \\
			&\Leftrightarrow \forall q \leq p ~ \exists \ap{\tau,s} \in \sigma_1 ~ 
			\exists r \leq q,s ~ r \Vdash_\kappa \sigma_0 = \tau \Leftrightarrow
			p \Vdash_\kappa \sigma_0\in\sigma_1
		\end{align*}
		using in the second to third equivalence the inductive assumption on the pairs $\ap{\rank(\sigma_0),\rank(\tau)}$ as $\tau$ ranges in $\dom(\sigma_1)$. The case of propositional connectives is easily handled, so we sketch the case $\phi \equiv \exists x \psi(x)$. Using Lemma \ref{lm:hkpresat},
		\[
		\begin{array}{lll}
			p \Vdash_\PP \phi^{\dot{H}_\kappa} &\Leftrightarrow& \bp{q \leq p: ~ \exists \tau \in V^\PP ~ 
			q \Vdash_\PP \psi(\tau)^{\dot{H}_\kappa} \wedge \tau \in \dot{H}_\kappa} \text{ is open dense} \\
			&\Leftrightarrow& \bp{q \leq p: ~ \exists \sigma_q \in H_\kappa\cap V^\PP ~ q \Vdash_\PP \psi(\sigma_q)^{\dot{H}_\kappa} } \text{ is open dense} \\
			&\Leftrightarrow& \bp{q \leq p: ~ \exists \sigma_q \in H_\kappa\cap V^\PP ~ q \Vdash_\kappa \psi(\sigma_q)^{\dot{H}_\kappa} } \text{ is open dense} \\
			&\Leftrightarrow& \cp{p \Vdash_\kappa \phi}^{H_\kappa}
		\end{array}
		\]
		where in the second to third equivalence we used that the intersection of two open dense sets is open dense. The thesis follows.
	\end{proof}

	\begin{lemma}[Lifting Lemma{, \cite[Lemma 17]{hamkins:resurrection_uplifting}}] \label{lm:lifting}
		Let $\ap{M,C} \prec \ap{M',C'}$ be transitive models of $\ZFC$ expanded with additional class predicates 
		$C$ and $C'$ (i.e., the inclusion map of $M$ in $M'$ extended with the assignment 
		$C\mapsto C'$ is elementary). 
		Let $\PP$ be  a definable class poset in $\ap{M,C}$ that is nice for forcing. 
		Let $\PP'$ be defined by the same formula in $\ap{M',C'}$ (obtained replacing $C$ with $C'$), 
		and suppose that $\PP'$ is also nice for forcing.

		Then for any $G$ $\ap{M,C}$-generic for $\PP$ and $G'$ $\ap{M',C'}$-generic for $\PP'$ such that $G' \cap M = G$, we have that $\ap{M[G],C,G} \prec \ap{M'[G'],C',G'}$.
	\end{lemma}

	We remark that since $\PP^{\Gamma,f}_\kappa$ is definable in $\ap{V_\kappa,f}$, the above results are applicable to this kind of iteration (even though the $\PP^{\Gamma,f}_\kappa$ we will be interested in are non-definable classes in $V_\kappa$ without making use of $f$).

	\begin{theorem} \label{thm:raconsistency}
		Let $\Gamma$ be weakly iterable and well behaved in a theory $T$ extending $\MK$ by a recursive set of first-order axioms. Then $\RA_\alpha(\Gamma)$ is consistent relative to the existence of an $(\alpha)$-uplifting cardinal $\kappa$ such that $V_{\kappa+1} \models T$ in a model of $\MK$.
	\end{theorem}
	\begin{proof}
		The proof follows the one of \cite[Thm.~18]{hamkins:resurrection_uplifting}.
		Let $V$ be the standard model of $\MK$. We prove by induction on $\alpha$ that $\PP_\kappa = \PP^{\Gamma,f}_\kappa$, the lottery iteration of $\Gamma$ relative to a function $f: \kappa \to \kappa$, forces $\RA_\alpha(\Gamma)$ whenever $f$ is Menas for $(\alpha)$-uplifting. Notice that such lottery iteration exists since $V_{\kappa+1} \models T$ and $\Gamma$ is weakly iterable in $T$. By Lemma \ref{lm:menas}, the existence of such an $f$ follows from the existence of an $(\alpha)$-uplifting cardinal, giving the desired result.

		Since $\RA_0(\Gamma)$ holds vacuously true, the thesis holds for $\alpha = 0$.
		Suppose now that $\alpha > 0$. Let $\dot{\QQ} \in V^{\PP_\kappa}$ be a name for a forcing in $\Gamma$, $\beta < \alpha$ be an ordinal. Using the Menas property for $f$, let $g: \lambda \to \lambda$ be such that $\ap{V_{\kappa+1}, f} \prec_{\Delta^1_1} \ap{V_{\lambda+1}, g}$, $g(\kappa) \geq \rank(\dot{\QQ})$ and $g$ is a Menas for $(\beta)$-uplifting function on $\lambda$. Let $\PP_\lambda = \PP^{\Gamma,g}_\lambda$ be the lottery iteration of $\Gamma$ relative to $g$. Notice that such lottery iteration exists since $V_{\lambda+1}$ models $\MK$ and $V_\kappa \prec V_\lambda$ implies that $V_\lambda$ models $T \setminus \MK$.
		
		Since $g \res \kappa = f$, by Proposition \ref{prop:lottery_iteration} we have that:
		\begin{itemize}
			\item $\PP_\kappa$ is ${<}\kappa$-cc and is definable in $\ap{V_\kappa,f}$, thus by Lemma \ref{lm:nice} is nice for forcing. Similarly, $\PP_\lambda$ is ${<}\lambda$-cc and definable in $\ap{V_\lambda,g}$, thus nice for forcing.
			\item $\PP_\kappa$ forces $2^\gamma \leq \kappa=\gamma^+$ and 
			$\PP_\lambda$ forces $2^\gamma \leq \lambda=\gamma^+$, since $\cpd(\Gamma)=\gamma$.
			\item Since $g(\kappa) > \rank(\dot{\QQ})$ and $\dot{\QQ}$ is in $\Gamma^{V^{\PP_\kappa}}$, $\PP_\lambda \res p \leq_\Gamma \PP_\kappa \ast \dot{\QQ}$ for a certain $p \in \PP_\lambda$.
		\end{itemize}
		Furthermore, by inductive hypothesis $\PP_\lambda$ forces $\RA_\beta(\Gamma)$. Thus, we only need to prove that $(H_{\gamma^+})^{V^{\PP_\kappa}} \prec (H_{\gamma^+})^{V^{\PP_\lambda}}$. The thesis will then follow by Proposition \ref{prop:radef}, since $\PP_\lambda \res p$ would be a legal (and winning) move in $\GGG^{\RA}$ after $\PP_\kappa \ast \dot{\QQ}$.

		Let $G$ be any $V$-generic filter for $\PP_\kappa$, $H$ be a $V[G]$-generic filter for $\val_G(\dot{\QQ})$. Since $g(\kappa) \geq \rank(\dot{\QQ})$, $\dot{\QQ}$ is one of the elements of the lottery sum considered at stage $\kappa + 1$ so that $G \ast H$ is $V$-generic for $\PP_\lambda \res \cp{\kappa+1}$. 
		Let $G'$ be $V[G \ast H]$-generic for $\PP_\lambda/G\ast H$.
		Since $\PP_\kappa$, $\PP_\lambda$ are nice for forcing in the respective models and $\ap{H_\kappa,f} \prec \ap{H_\lambda,f}$, we can apply Lemma \ref{lm:lifting} to obtain that $H_\kappa[G] \prec H_\lambda[G \ast H \ast G']$. Furthermore by Lemma \ref{lm:hkpresat},
		\[
		H_\kappa^{V[G]} = H_\kappa[G] \prec H_\lambda[G \ast H \ast G'] = H_\lambda^{V[G \ast H \ast G']}
		\]
		Since $\gamma^+ = \kappa$ in $V[G]$ and $\gamma^+ = \lambda$ in $V[G \ast H \ast G']$, the proof
		is completed.
	\end{proof}
	
	For classes $\Gamma$ that are iterable in $\MK$ alone (i.e., $\ALL$, axiom-$A$, proper, $\SP$, ${<}\kappa$-closed), the above result gives consistency relative to an $(\alpha)$-uplifting cardinal existing in a model of $\MK$, which in turn follows from the existence of a Mahlo cardinal (see Proposition \ref{pr:upliftingmahlo}). For $\Gamma = \SSP$, which is iterable assuming the existence of a proper class of supercompact cardinals, we obtain consistency relative to the existence of a Mahlo cardinal which is a limit of supercompact cardinals.
	This result does not fully match the best known upper bound for the consistency strength of $\RA(\SSP)$ given in \cite[Thm. 3.1]{tsaprounis:on_resurrection}, where it is shown that in the presence of class-many Woodin cardinals $\MM^{++}$ implies $\RA(\SSP)$ (according to Hamkins and Johnstone's terminology). However, it is not clear how this latter result can be generalized to the axioms $\RA_\alpha(\SSP)$ we introduced.

	The first author has shown in \cite{a:thesis} that we can get generic absoluteness for ${<}\kappa$-closed forcings simultaneously for all cardinal $\kappa$.

	\begin{definition}
		Let $\BB$ be a complete Boolean algebra. We define its closure degree $\cd(\BB)$ as the largest cardinal $\kappa$ such that there exists a dense subset $D \subseteq \BB$ which is a ${<}\kappa$-closed poset.
		We denote as $\Gamma_\kappa$ the class of $\BB$ with $\cd(\BB)=\kappa$.
	\end{definition}
	
	\begin{theorem}[Audrito~\cite{a:thesis}] \label{thm:rak_everywhere}
		It is consistent relative to a Mahlo cardinal that the axioms $\RA_\ON(\Gamma_\kappa)$ hold simultaneously for all cardinals $\kappa$.
	\end{theorem}
	If we first force $\RA_{\ON}(\SSP)$ over a model with class-many Woodin cardinals (see Proposition \ref{prop:woodinrall}) and then apply the above result restricted to $\kappa \geq \omega_2$, we also get the consistency of $\RA_\ON(\ALL) + \RA_{\ON}(\SSP) + \forall \kappa \geq \omega_2 \RA_\ON(\Gamma_\kappa)$.

	\section{Concluding remarks and open problems} \label{sec:resurrection_open}

First of all we want to bring forward a few remarks regarding the philosophical import of the iterated resurrection axioms. There is a surprising and converging set of arguments which outline that forcing axioms, generic absoluteness results, Baire's category theorem, and the axiom of choice are different sides of the same coin: each of these principles provide non-constructive methods which can be employed successfully in the analysis of mathematical problems. 
A detailed account of the web of relations, implications, and equivalences between these different principles can be found in~\cite{a:viale:useful_axioms}. For the moment let us remark that Todor\v{c}evi\'c has shown that the axiom of choice can be equivalently formulated over the theory $\ZF$ as the assertion that $\FA_\kappa(\Gamma_\kappa)$ holds for all cardinals $\kappa$ (where $\Gamma_\kappa$ is the class of ${<}\kappa$-closed forcings, see~\cite{a:viale:useful_axioms}).
In particular any forcing axiom implying $\FA_\kappa(\Gamma_\kappa)$ can be regarded as a (local) strengthening of the axiom of choice.
In this regard the axioms $\RA_\alpha(\Gamma)$ for $\Gamma\supseteq\Gamma_\kappa$ appear to be natural companions of the axiom of choice, while the axioms $\RA_\ON(\ALL)$ and $\RA_\ON(\SSP)+\MM$ (or $\MM^{+++}$) 
are natural maximal strengthenings 
of the axiom of choice at the levels $\omega$ and $\omega_1$.
Hence it is in our opinion natural to try to isolate classes of forcings $\Delta_\kappa$ as $\kappa$ ranges among the cardinals such that:
\begin{enumerate}[(a)]
\item $\kappa=\cpd(\Delta_\kappa)$ for all $\kappa$.
\item $\Delta_\kappa\supseteq \Gamma_\kappa$ for all $\kappa$.
\item $\FA_\kappa(\Delta_\kappa)$ and $\RA_\omega(\Delta_\kappa)$ are simultaneously consistent for all $\kappa$.
\item For all cardinals $\kappa$, $\Delta_\kappa$ is the largest possible $\Gamma$ with $\cpd(\Gamma)=\kappa$ for which  $\FA_\kappa(\Delta_\kappa)$ and $\RA_\omega(\Delta_\kappa)$ are simultaneously 
consistent (and if possible for all $\kappa$ simultaneously). 
\end{enumerate}
Compare the above requests with requirements (3) and (4) in the discussion motivating the introduction of the iterated resurrection axioms
on page~\pageref{aim:resurrection-1}. In this regard it appears that we have now a completely satisfactory answer on what are $\Delta_{\omega}$ and $\Delta_{\omega_1}$: i.e., respectively the class of \emph{all} forcing notions and the class of all $\SSP$-forcing notions.

\begin{question}
What is $\Delta_{\omega_2}$ (or more generally $\Delta_{\kappa}$ for $\kappa>\omega_1$)? 
Which criteria can bring us to isolate it?
\end{question}
The results of this paper outline that any interesting iteration theorem for a class $\Gamma\supseteq \Gamma_{\omega_2}$ closed under two step iterations, can be used to prove that $\RA_\ON(\Gamma)$ is consistent relative to suitable large cardinal assumptions and that it freezes the theory of $H_{\omega_3}$ with respect to forcings in $\Gamma$ preserving $\RA_\omega(\Gamma)$.
It is nonetheless still a mystery which classes $\Gamma\supseteq\Gamma_{\omega_2}$ can give us
a nice iteration theorem, even if the recent works, by Neeman, Asper\`o, Krueger,
Mota, Velickovic and others are starting to shed some light on this problem (see among others ~\cite{krueger:forcingclubaleph2,krueger:mota:stronglyproper,neeman:forcingaleph2}).

We can dare to be more ambitious and replicate the above type of issue at a much higher level of the set theoretic hierarchy.
There is a growing set of results regarding the first-order theory of $L(V_{\lambda+1})$ assuming $\lambda$ is a very large cardinal
(i.e., for example admitting an elementary $j: L(V_{\lambda+1})\to L(V_{\lambda+1})$ 
with critical point smaller than $\lambda$, see among others~\cite{dimonte:I0GCH,dimonte:I0,woodin:beyondI0}). 
It appears that large fragments of this theory are generically invariant with respect to a great variety of forcings.

\begin{question}
Assume $j:L(V_{\lambda+1})\to L(V_{\lambda+1})$ is
elementary with critical point smaller than $\lambda$ .
Can any of the results of the present paper be of use in the study of which type of 
generic absoluteness results may hold at the level of $L(V_{\lambda+1})$?
\end{question}

Next, we address the possibility of an extension of our results on \emph{clopen} class games and iterated resurrection to \emph{open} class games. In fact, open class games are also provably determined assuming $\MK$. It is therefore natural to inquire what are the consequences of the corresponding resurrection axioms and whether their consistency can be proved along the same lines.

\begin{definition}
	The long $\Gamma$-resurrection game $\GGG^{\RA_\infty}$ is the game of length $\omega$ defined as follows. Player I (\emph{Kill}) plays Boolean algebras $\BB_{n+1}$ such that $\BB_{n+1} \leq_\Gamma \CC_n$. Player II (\emph{Resurrect}) plays Boolean algebras $\CC_n$ such that $H_{\gamma^+}^{\CC_n} \prec H_{\gamma^+}^{\CC_{n+1}}$ and $\CC_n \leq_\Gamma \BB_n$. Player II loses if it cannot move, and wins if the game reaches full length.
	\[
	\xymatrixcolsep{1pc}
	
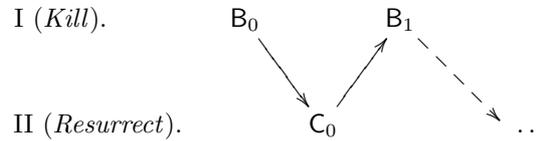
\xymatrix{
	\text{I (\emph{Kill}).\phantom{aaaaai}} &  \BB_0 \ar@{->}[rd] & & \BB_1 \ar@{-->}[rd] & \\
	\text{II (\emph{Resurrect}).} & & \CC_0 \ar@{->}[ru] & & \quad \quad \ldots
	}
	\]
	$\RA_\infty(\Gamma)$ is the assertion that Player II (\emph{Resurrect}) wins the long $\Gamma$-resurrection game after $\ap{\2, \2}$.
\end{definition}

Since the game  $\GGG^{\RA_\infty}$ is open, it is still determined in $\MK$ and a theory similar to that of $\GGG^{\RA}$ can be carried out. Furthermore, the definition of $(\infty)$-uplifting cardinal can be given and proved to follow from a Mahlo cardinal, while an easy adaptation of the results in this paper can show that $\RA_{\infty}(\Gamma)$ is consistent relative to a Menas for $(\infty)$-uplifting cardinal. However, it is not clear how to obtain a Menas for $(\infty)$-uplifting cardinal from an $(\infty)$-uplifting cardinal.

\begin{question}
	What is the consistency strength of $\RA_{\infty}(\Gamma)$?
	Can $\RA_{\infty}(\Gamma)$ or $\RA_{\ON}(\Gamma)$ entail stronger generic absoluteness properties than $\RA_\omega(\Gamma)$?
\end{question}

Finally the present results leave a somewhat narrow gap between the upper bounds for the iterated resurrection axioms and their lower bounds. Hence the following question is still open.

\begin{question}
	What is the exact consistency strength of $\RA_\ON(\Gamma)$ for the different classes $\Gamma$ for which the axiom is consistent? Are these axioms consistency-wise strictly stronger than
	Hamkins and Johnstone resurrection axioms in terms of their large cardinal strength?
\end{question}

\subsection*{Acknowledgements} 
The second author acknowledges support from the PRIN2012 Grant ``Logic, Models and Sets"
(2012LZEBFL), the GNSAGA,
and the Junior PI San Paolo grant 2012 NPOI (TO-Call1-2012-0076).
This research was completed whilst the second author was a visiting fellow at the 
Isaac Newton Institute for Mathematical Sciences in the programme ``Mathematical, 
Foundational and Computational Aspects of the Higher Infinite'' (HIF) funded by EPSRC grant EP/K032208/1. 

	\bibliographystyle{plain}
	\bibliography{Bibliography}
\end{document}